\documentclass[12pt]{amsart}

\usepackage{amsfonts,amsthm,amsmath,amssymb,amscd,mathrsfs,tikz}
\usepackage{graphics}
\usepackage{indentfirst}
\usepackage{cite}
\usepackage{latexsym}
\usepackage[dvips]{epsfig}

\usepackage{color}
\usepackage{amssymb,amsmath,bm}
\usepackage{fontspec}

\usepackage[title]{appendix}
\usepackage{enumerate}

\newtheorem{theorem}{Theorem}[section]
\newtheorem{remark}{Remark}[section]
\newtheorem{definition}{Definition}[section]
\newtheorem{lemma}[theorem]{Lemma}
\newtheorem{pro}[theorem]{Proposition}

\renewcommand{\div}{{\rm div }}

\renewcommand{\r}{\mathbb{R}}

\newcommand{\bt}{\begin{theorem}}
	\newcommand{\bl}{\begin{lemma}}
		\newcommand{\el}{\end{lemma}}
	\newcommand{\et}{\end{theorem}}

\newcommand{\bn}{\begin{eqnarray}}
	\newcommand{\en}{\end{eqnarray}}
\newcommand{\bnn}{\begin{eqnarray*}}
	\newcommand{\enn}{\end{eqnarray*}}

\newcommand{\ba}{\begin{aligned}}
	\newcommand{\ea}{\end{aligned}}
\newcommand{\be}{\begin{equation}}
	\newcommand{\ee}{\end{equation}}

\newcommand{\p}{\partial}

\newcommand{\dd}{\mathrm{d}}

\newcommand{\bBV}{\boldsymbol{V}}

\newcommand{\Bu}{{\boldsymbol{u}}}
\newcommand{\Be}{{\boldsymbol{e}}}

\newcommand{\D}{{\boldsymbol{D}}}

\newcommand{\BPsi}{{\boldsymbol{\Psi}}}

\newcommand{\BA}{{\boldsymbol{A}}}
\newcommand{\Bx}{{\boldsymbol{x}}}

\newcommand{\Bp}{{\boldsymbol{\phi}}}

\newcommand{\OR}{{\mathscr{O}}_{R}}

\begin{document}
	
	\title[Liouville-type theorems]
	{ Liouville-type theorems for  the stationary non-Newtonian fluids in a slab}
	
		\author{Jingwen Han}
		\address{School of Mathematics and Statistics, Anhui Normal University, Wuhu, China}
		\email{mathjwh95@ahnu.edu.cn}

    \author{Han Li}
\address{School of Mathematical Sciences, Soochow University, No. 1 Shizi Street, Suzhou, China}
\email{hli3@stu.suda.edu.cn}

	\begin{abstract}
		In this paper, we investigate the Liouville-type theorems for stationary solutions to the shear thickening fluid equations in a slab. We show that  the axisymmetric solution  must be trivial if its local $L^{\infty}$-norm grows mildly as the radius $R$ grows. Also a bounded general solution $\Bu$ must be trivial if  $ru^{r}$ is bounded. The proof is inspired by \cite{aBGWX} for the Navier-Stokes equations and the key point is to establish a Saint-Venant type estimate that characterizes the growth of local Dirichlet integral  of nontrivial solutions. One new ingredient is the estimate of the constant in Korn's inequality over different domains.
	\end{abstract}


	
	\maketitle

	\section{Introduction and Main Results}
	In this paper, we are concerned with the Liouville-type theorem for  the stationary non-Newtonian fluid equations in a slab domain $\Omega=\mathbb{R}^{2}\times [0,1]$, i.e., 
	\begin{equation}\label{eqsteadynNfs}
		\left\{ \ba
		& -\text{div} \BA_{p}(\Bu) + (\Bu \cdot \nabla )\Bu + \nabla P  = 0, \ \ \ \ \
		&\mbox{in}\ \Omega, \\
		& \nabla \cdot \Bu =0,  \ \ \ \ \ &\mbox{in}\ \Omega.  \\
		\ea \right.
	\end{equation}
	Here the unknown function $\Bu=(u^1,u^2,u^3)$  is the velocity field of the fluid,  $P$ is the pressure. The diffusion term 
	\[
	\BA_{p}(\Bu)=(1+|\D(\Bu)|^2)^{\frac{p-2}{2}}\D(\Bu), \quad 1<p<+\infty,
	\]
	where	$\D(\Bu)=\frac{1}{2}(\nabla\Bu+(\nabla\Bu)^{\top})$ is the stress tensor. When $p=2$, the equations \eqref{eqsteadynNfs} become the classical Navier-Stokes equations for Newtonian fluids.  When $2 <p< +\infty$, the equations describe the motion of shear thickening fluids, for which the viscosity increases along with the shear rate $| \D(\Bu)|$. When $1 <p< 2$, the equations correspond to the shear thinning fluids.  We refer to \cite{WL} for the physical background of the equations. In this paper, we focus on the equations for shear thickening fluids, i.e., $p \ge 2$,  which are   supplemented with no-slip boundary conditions
	\begin{equation}\label{noslipboun}
		\Bu=0, \ \ \ \text{at}\ x_{3}=0\ \text{and}\ 1.
	\end{equation}  
	
	The existence of weak solutions  in bounded domains has been shown by Ladyzhenskaya \cite{OAL1,OAL2,OAL3}, Lions \cite{LJL} for $p \ge \frac{3n}{n+2}$,  Frehse-M\'{a}lek-Steinhauer\cite{FJMJSM} and Ruzi\v{c}k\v{a}\cite{MR} for $ p \ge \frac{2n}{n+1}$. In unbounded domains with noncompact boundaries, Yang and Yin   \cite{yangYin18} obtained the existence and uniqueness of weak solutions.  We can also refer to  \cite{MP,FJMJSM} for partial regularity results of weak solutions.

	The study of the Liouville-type theorems for the stationary  Navier-Stokes equations was pioneered to investigate the uniqueness of D-solutions in a domain $\Omega$, i.e., the solution satisfies the finite integral
	\begin{equation}\label{finiDinNS}
		\int_{\Omega}|\nabla\Bu|^{2}\,\dd\Bx<\infty.
	\end{equation}
	Whether the D-solutions $\Bu$ in the three-dimensional whole space equals zero if it vanishes at infinity remains an open and challenging problem \cite[X.9 Remark X. 9.4]{GAGP11}.
	Gilbarg and Weinberger \cite{GWASP78} proved the two-dimensional D-solutions of the Navier-Stokes equations in the whole space must be constant. But the general three-dimensional case is not known yet, it was obtained by Korobkov, Pileckas and Russo \cite{KPRJMFM15}  that the axisymmetric without swirl D-solutions of the Navier-Stokes equations in $\mathbb{R}^{3}$ must vanish.  In the slab domain, Carrillo, Pan, Zhang and Zhao \cite{BPZZ} showed the D-solutions supplemented with no-slip boundary conditions must be zero. Under the integrability assumption, Tsai \cite{TT21}  obtained the Liouville-type theorem for the Navier-Stokes equations in the slab domain.
	One may refer \cite{CCMP14,CW2019CVPDE,CWDCDS16,SN16, KTWJFA17,WJDE19,NZ19} and the references therein for partial progress for the problem.   Another classification of $L^{\infty}$-solutions to the Navier-Stokes equations  is also a significant area of Liouville-type theorems, which plays an important role in analyzing the singularity and the far field behavior for the solutions. The remarkble progress by  Koch, Nadirashvili, Seregin and Sverak \cite{KNSS09} showed that any bounded solutions of the two-dimensional or axisymmetric  with no swirl to the Navier-Stokes equations must be constant vectors, they also showed the axisymmetric bounded solutions  with $r|\Bu|\leq C$ must be trivial. Recently, Bang, Gui, Wang and Xie  \cite{aBGWX} considered the bounded solutions in the slab, 
	they proved in the slab with no slip boundary conditions that any bounded solution of the Navier-Stokes equations is trivial if it is axisymmetric or $ r u^r$ is bounded for general  solutions.
	
	Compared with the  Navier-Stokes equations, there are few Liouville-type results on the non-Newtonian fluids equations. Bildhauer, Fuchs and Zhang in  \cite{JMFM13bfz} considered Liouville-type theorems for  the  stationary generalized Newtonian fluids in the two-dimensional whole space.    Specifically, for shear thinning  fluids, they proved  any bounded solutions must be a constant vector. While for shear thickening fluids, they  proved the solutions must be trivial under the decay assumptions.  In \cite{BK},  Jin and  Kang derived the Liouville-type theorems  for weak solutions with finite energy. Chae, Kim and  Wolf \cite{DJ1,DJJDE21}  assert that if a suitable weak solution   $\Bu \in W^{1,p}(\mathbb{R}^3)$  of the non-Newtonian fluids satisfies some extra integrability conditions, then $\Bu$ is trivial. For more references about Liouville-type theorems of the non-Newtonian fluids in $\r^2$, one may refer to
	\cite{MF,MFGZ,GZ} and the references therein. To the best of our knowledge, there is no Liouville-type result for the non-Newtonian fluids in a slab.
	
	Let the standard cylindrical coordinates of
	$\mathbb{R}^3$ are $(r,\theta, z)$, which are defined as follows:
	\[
	\Bx=(x_{1}, x_{2}, x_{3})=(r\cos\theta, r\sin\theta, z).
	\]
	In cylindrical coordinates, the general   velocity $\Bu$  can be written as 
	\begin{equation}\label{cylindcoo}
		\Bu=u^{r}(r,\theta,z)\Be_{r}+u^{\theta}(r,\theta,z)\Be_{\theta}+u^{z}(r,\theta,z)\Be_{z},
	\end{equation}
	where scalar components $u^{r}, u^{\theta}, u^{z}$ are called radial, swirl and axial velocity, respectively, and the basis vectors $\Be_{r}, \Be_{\theta}, \Be_{z}$ are
	\[
	\Be_{r}=(\cos\theta, \sin\theta, 0), \quad
	\Be_{\theta}=(-\sin\theta, \cos\theta,0), \quad
	\Be_{z}=(0,0,1).
	\]
	If a scalar function $\varphi$ does not depend on $\theta$, it is called axisymmetric. If the scalar components $u^{r}, u^{\theta}, u^{z}$  in \eqref{cylindcoo} do not depend on $\theta$, then the velocity $\Bu$ is called axisymmetric. 
	In this paper, we investigate the Liouville-type theorems  of axisymmetric and general  shear-thickening fluids in a slab.  
	
	Before stating the main results of this paper, we first give the definition of  the weak solution to the stationary non-Newtonian fluid equations in a slab domain with no-slip boundary conditions.

	\begin{definition}\label{def1we}
		Let $1< p < +\infty$. A vector field $\Bu \in H^1_{loc}(\Omega)\cap W^{1,p}_{loc}(\Omega)$ is called the weak solution of  the equations \eqref{eqsteadynNfs} with no-slip boundary conditions \eqref{noslipboun},  if $\Bu$ satisfies
		\begin{equation}
			\begin{aligned}
				\int_{\Omega} \BA_{p}(\Bu):\D(\Bp) +(\Bu \cdot \nabla )\Bu \cdot \Bp \,\dd \Bx=0,  \quad \quad \text{for any} \,\,\Bp \in C_{0}^{\infty}(\Omega)\,\, \text{with} \,\,\text{\div}\Bp=0.
			\end{aligned}
		\end{equation}		
	\end{definition}
According to Galdi \cite[Theorem III. 5.3]{GAGP11}, there is an associated pressure $P\in L_{loc}^{p}(\Omega)$ such that the weak solution  $\Bu$ in Definition \ref{def1we}  satisfies
\begin{equation}
    \begin{aligned}
        \int_{\Omega} \BA_{p}(\Bu):\D(\Bp) +(\Bu \cdot \nabla )\Bu \cdot \Bp \,\dd \Bx-\int_{\Omega}P\nabla\cdot \Bp\, \dd \Bx=0,  \quad \quad \text{for any} \,\,\Bp\in C_{0}^{\infty}(\Omega).
    \end{aligned}
\end{equation}

 
	The first main result of this paper is stated as follows.
  \begin{theorem}\label{DS}
      Assume that $2\leq p<+\infty$ and $\Omega= \mathbb{R}^2\times (0, 1)$. Let $\Bu$ be the weak solution of the equations \eqref{eqsteadynNfs} in a slab with no-slip boundary conditions \eqref{noslipboun},   then $\Bu\equiv 0$, provided that the Dirichlet integral is finite, i.e., 
      \[
  \int_{0}^{1}\iint_{\mathbb{R}^{2}}(|\nabla\Bu|^{2}+|\D(\Bu)|^p)\,\dd x_{1}\dd x_{2}\dd x_{3}<\infty.
      \]
 \end{theorem}	
\begin{remark}
		When $p=2$,  it is  the Liouville-type  theorem for D-solutions to the Navier-Stokes equations in \cite{BPZZ}.
		In this sense, our Theorem \ref{DS} is a generalization of Theorem 1.1 in \cite{BPZZ}.
	\end{remark}
    
	Then, we show the Liouville-type theorem for the axisymmetric solution in a slab.
	\begin{theorem}\label{th:01}
		Assume that $2\leq p<+\infty$ and $\Omega= \mathbb{R}^2\times (0, 1)$. Let $\Bu$ be an axisymmetric weak solution of the equations \eqref{eqsteadynNfs} in a slab $\Omega=\mathbb{R}^{2}\times (0,1)$ with no-slip boundary conditions \eqref{noslipboun}.	
		
		
		
		$(a_1)$ For $2\leq p<3$, if $\Bu$ satisfies
		\begin{equation}\label{LIes1}
			\lim\limits_{R\rightarrow +\infty} R^{-\frac{1}{3-p}}\sup_{z\in[0,1]}|\Bu(R,z)|=0,
		\end{equation}
		then $\Bu \equiv 0$.
        
		$(a_2)$ For $p\ge 3$, there exists a constant $C(p) > 0$, such that if $\Bu$ satisfies
		
		\begin{equation}\label{LIes2}
			\lim\limits_{R\rightarrow +\infty} e^{-\frac{ R}{C(p)}}\sup_{z\in[0,1]}|\Bu(R,z)|=0,
		\end{equation}		
        then $\Bu \equiv 0$.
    
	\end{theorem}
	
	We have the following remark on Theorem \ref{th:01}.
	\begin{remark}
		If  $p=2$, then $\frac{1}{3-p}=1$, Theorem \ref{th:01} is consistent with Theorem 1.1 in \cite{aBGWX} for the Navier-Stokes equations.
		In this sense, our Theorem \ref{th:01} is a generalization of Theorem 1.1 in \cite{aBGWX}.
	\end{remark}
	

	Next, we establish a Liouville-type theorem for general flows in the slab.
	
	\begin{theorem}\label{th:02}		
		Let $\Bu$ be a weak solution to the equations \eqref{eqsteadynNfs} in $\Omega=\mathbb{R}^2 \times (0,1)$ with no-slip boundary conditions \eqref{noslipboun}. Then  $\Bu \equiv 0$ if 
    	\begin{equation}\label{3Db}
			\sup_{(r,\theta,z)\in \Omega}|\Bu(r,\theta,z)| < \infty \quad and \quad \sup_{(r,\theta,z)\in \Omega} r|u^r(r,\theta,z)| < \infty.
		\end{equation}
	\end{theorem}
        
		


	Now we outline  the proof for the main results.  Inspired by the work  \cite{aBGWX}, we establish various Saint-Venant estimates for the Dirichlet integral of  $\Bu$, i.e. $\int (|\nabla\Bu|^{2}+|\D(\Bu)|^p)\,\dd \Bx$, over  finite subdomains and characterize the growth of the nontrivial solutions. One new ingredient is the estimate for the constant in Korn's inequality over different domains. The Saint-Venant's principle was initially used in \cite{JKKTARMA66,RATARMA65} to study the solutions of elastic equations.  Then it was applied to study the well-posedness of the solutions to the stationary Navier-Stokes equations in a cylinder\cite{LSZNSL80}. Recently, the idea was adapted to prove the Liouville-type theorems for the solutions of  stationary Navier-Stokes equations (cf. \cite{aBGWX,aBYA,KTWJFM24}).

	The rest of this paper is organized as follows. In Section  \ref{Sec2}, we introduce the Bogovskii map,  some notations and inequalities. The Liouville-type theorems for axisymmetric solutions are proved in Section \ref{Sec3}. Section \ref{Sec4}  is devoted to the Liouville-type theorems for general solutions.

	\section{PRELIMINARIES}\label{Sec2} 
	Some  notations are given below. 
	We denote the cut-off domain $D_{R}=(R-1,R)\times (0,1)$, $\mathcal{D}_R =(R-1, R)\times(0, 2\pi)\times(0, 1)$, $\OR=(B_{R}\setminus \overline{B_{R-1}})\times (0,1)$, $Z_R=(B_{R}\setminus \overline{B_{\frac{R}{2}}})\times (0,1)$ and $Q_R=(B_{R}\setminus \overline{B_{\frac{R}{2}}})\times (0,R)$, where $B_{R}=\left\{(x_{1},x_{2})\in\mathbb{R}^{2}: x^{2}_{1}+x^{2}_{2}<R^2\right\}$.
	
	\begin{definition}
		Let $1\leq p \leq +\infty$ and  $D $ be a bounded domain in $\mathbb{R}^n$. Denote
		\[
		L^p_0(D)=\left\{v(\Bx):v\in L^p(D), \int_D v(\Bx)\,\dd \Bx=0\right\}. 
		\]
	\end{definition}
	
	Then we introduce the Bogovskii map, which gives a solution to the divergence equation.
	\begin{lemma}\label{Bogovskii}
		Let $D$ be a bounded Lipschitz domain in $\mathbb{R}^n$, $n\geq 2 $.  For any $p\in (1, +\infty)$, there is a linear map $\boldsymbol{\Phi}$ that maps a scalar function $g\in L^p_0(D)$ to a vector field $\bBV = \boldsymbol{\Phi} g \in W_0^{1, p}(D; \mathbb{R}^n)$ satisfying
		\be \nonumber
		{\rm div}~\bBV = g \ \text{in}\,\, D \quad \text{and} \quad \|\bBV\|_{W_{0}^{1, p}(D)} \leq C (D, p) \|g\|_{L^p(D)}.
		\ee
		
		In particular,
		\begin{enumerate}
			\item   For any $g \in L^p_0(D_R)$,
			the  vector valued function $\bBV = \boldsymbol{\Phi} g \in W_0^{1,p}(D_R; \mathbb{R}^2)$ satisfies
			\be \nonumber
			\partial_r V^r + \partial_z V^z =g \ \  \mbox{in}\,\, D_R
			\quad
			\text{and}
			\quad
			\|\widetilde{\nabla } \bBV\|_{L^p(D_R)}
			\leq C \|g\|_{L^p(D_R)},
			\ee
			where $\widetilde{\nabla } = (\partial_r, \ \partial_z ) $ and $C$ is a constant independent of $R$.
			
			\item For any $g \in L_0^p(\mathcal{D}_R)$,
			the   vector valued function $\bBV = \boldsymbol{\Phi} g \in W_0^{1,p}( \mathcal{D}_R; \mathbb{R}^3)$ satisfies
			\be \nonumber
			\partial_r V^r + \partial_\theta V^\theta +  \partial_z V^z =g \ \   \mbox{in}\,\, \mathcal{D}_R
			\quad
			\text{and}
			\quad
			\|\overline{\nabla } \bBV\|_{L^p(\mathcal{D}_R)}
			\leq C \|g\|_{L^p( \mathcal{D}_R)},
			\ee
			where $\overline{\nabla} = (\partial_r, \ \partial_\theta, \ \partial_z ) $ and $C$ is a constant independent of $R$.
			\item  For any $g \in L_0^p(Z_R)$,
			the vector valued function $\bBV = \boldsymbol{\Phi} g \in W_0^{1,p}(Z_R ; \mathbb{R}^3)$ satisfies
			\be \nonumber
			{\rm div}~\bBV = g \ \   \mbox{in}\,\, Z_R
			\quad
			\text{and}
			\quad
			\|\nabla \bBV\|_{L^p(Z_R)}
			\leq CR \|g\|_{L^p(Z_R)},
			\ee
			where $C$ is a constant independent of $R$.
		\end{enumerate}

		Part (1) and (2) can be found in  \cite[Lemma 2.1]{aBGWX},  Part (3) is given in \cite[Proposition 2.1]{BPZZ}.
    Its general form is due to    Bogovskii\cite{BM}, see \cite[Section III.3]{GAGP11} and \cite[Section 2.8]{TT18}.
		\end{lemma}

	Next, we give the Korn's inequality due to V.A. Kondrat'ev and O.A. Oleinik\cite[Theorem 4]{OAVA}. 
\begin{lemma}\label{Korn F}
Let $ \Bu \in W^{1,p}(D)$, assume that the domain \(D\) has a Lipschitz boundary, and that \( \Bu \in V \), where the linear space \( V \), weakly closed in \( W^{1,p}(D) \), is such that \( V \cap \mathfrak{M} = \{0\} \), where \(\mathfrak{M}\) is the set of rigid translations, that is, functions of the type \( A\Bx + B \); \( A \) is a skew-symmetric matrix with constant entries and \( B \) is a constant vector. Then for \( \Bu(x) \) we have a Korn inequality of the form
\begin{equation}\label{kk}
\|\nabla\Bu\|_{L^p(D)}	\leq C\|\D(\Bu)\|_{L^{p}(D)}, \quad 1 < p < +\infty,
\end{equation}
where the constant \( C \) does not depend on \( \Bu \).
\end{lemma}

Subsequently, we give a particular version of  Korn's inequality for the axisymmetric  and general vector fields in the cylindrical domains, which will be used frequently in  
Section  \ref{Sec3} and Section  \ref{Sec4}.
\begin{lemma}[Korn's inequality in  the cylindrical domains]\label{korn}
\, \\
$\text{Case} (1) $. Assume that $\Bu $ is an axisymmetric vector field, which belongs to $W^{1,p}(\OR)$, and $\Bu=0 $ on $\Gamma_1=\{(x_1,x_2,x_3)\in \p \OR:x_3=0,1 \}$. Then there exist some constants $R_0>0$ and $C>0$, such that when $R \ge R_0$, it holds that 
\begin{equation}\label{yski}
    \begin{aligned}
      \| \nabla \Bu \|_{L^p(\OR)} \leq C \| \D(\Bu) \|_{L^p(\OR)}. 
    \end{aligned}
\end{equation}
Here $C$ is a universal constant independent of $R$ and $\Bu$.\\
$\text{Case} (2)$. Assume that $\Bu \in W^{1,p}(Q_R)$, and $\Bu=0 $ on $\Gamma_2=\{(x_1,x_2,x_3)\in \p Q_R:x_3=0,R \}$. Then there exists a constant $C>0$, which is independent of  $R$ and $\Bu$, such that 
		\begin{equation}\label{ki2}
			\|\nabla\Bu\|_{L^p(Q_R)}\leq C\|\D(\Bu)\|_{L^{p}(Q_R)}.
		\end{equation}
\end{lemma}

\begin{proof}Proof of $\text{Case}(1)$. Step 1.
Let $D=\mathcal{D}_R$, one can verify that Lemma \ref{Korn F} are satisfied:
Define the space 
   \[
   V = \{ u \in W^{1,p}(\mathcal{D}_R) : u = 0 \text{ on } \,\Gamma=\{(x_1,x_2,x_3\in \p \mathcal{D}_R:x_3=0,1)\}\}.
   \]
   Since  $V$ is a closed subspace 
   of $W^{1,p}(\mathcal{D}_R)$ and hence weakly closed.
   
   Then  one shows that $V \cap \mathfrak{M} = \{0\}$, where $\mathfrak{M}$ is the space of 
   rigid translations. Let $\Bu(\Bx) = A\Bx + B \in V \cap \mathfrak{M}$, with skew-symmetric matrix $A$  
   and constant vector $B$. The boundary condition $\Bu = 0$ on $x_3 = 0$ and $x_3 = 1$ implies:
   \begin{itemize}
     \item  On  $x_3 = 0$: $A(x_1,x_2,0)^\top + B = 0$ for all $(x_1,x_2)$.
     \item On $x_3 = 1$: $A(x_1,x_2,1)^\top + B = 0$ for all $(x_1,x_2)$.
   \end{itemize}
   Subtracting the  equations  gives $A(0,0,1)^\top = 0$ for all $(x_1,x_2)$, which forces the 
   third column of $A$ to be zero. Since $A$ is skew-symmetric, this implies $A = 0$. 
   Then from the first equation, $B = 0$. Hence $\Bu = 0$. Since all conditions of Lemma \ref{Korn F} are satisfied, the Korn's inequality \eqref{kk} holds in $\mathcal{D}_R$.

Step 2. By Step 1, Korn’s inequality \eqref{kk} holds in $\mathcal{D}_R$. Consequently, one has 
		\begin{equation}\label{Kororv}
     \|\overline{\nabla}\Bu\|_{L^p(\mathcal{D}_R)}\leq C\|\overline{\D}(\Bu)\|_{L^{p}(\mathcal{D}_R)},
		\end{equation}
		where  $\overline{\nabla}=(\p_r,\p_{\theta},\p_z)$, and $\overline{\D}(\Bu)=\frac{1}{2}\left(\overline{\nabla}+\overline{\nabla}^T\right)\Bu $, and $C$ is a universal constant  independent of $R$. 
		
		Since $\Bu$ is axisymmetric, it holds that 
        \be
\begin{aligned}
    |\nabla \Bu|^2&=\left|\p_r u^r\right|^2+\left|\p_r u^{\theta}\right|^2+\left|\p_r u^z\right|^2+\left|\p_z u^r\right|^2+\left|\p_z u^{\theta}\right|^2+\left|\p_z u^z\right|^2+\left|\frac{u^r}{r}\right|^2+\left| \frac{u^{\theta}}{r}\right|^2,\\
    |\D(\Bu)|^2&=\left|\p_r u^r\right|^2+\frac{1}{2}\left|\p_r u^{\theta}-\frac{u^{\theta}}{r}\right|^2+\frac{1}{2}\left|\p_z u^r+\p_r u^z\right|^2+\frac{1}{2}\left|\p_z u^{\theta}\right|^2+\left|\p_z u^z\right|^2+\left|\frac{u^r}{r}\right|^2, \\
     |\overline{\nabla} \Bu|^2&=\left|\p_r u^r\right|^2+\left|\p_r u^{\theta}\right|^2+\left|\p_r u^z\right|^2+\left|\p_z u^r\right|^2+\left|\p_z u^{\theta}\right|^2+\left|\p_z u^z\right|^2,\\
       |\overline{\D}(\Bu)|^2&=\left|\p_r u^r\right|^2+\frac{1}{2}\left|\p_r u^{\theta}\right|^2+\frac{1}{2}\left|\p_z u^r+\p_r u^z\right|^2+\frac{1}{2}\left|\p_z u^{\theta}\right|^2+\left|\p_z u^z\right|^2.
\end{aligned}
        \ee
By straightforward computations, one has
		\be\label{ineqnaDu}
		\begin{aligned}
			|\overline{\nabla}\Bu|^2&=|\nabla \Bu|^2-\left|\frac{u^r}{r}\right|^2-\left| \frac{u^{\theta}}{r}\right|^2,\\
			|\overline{\D}(\Bu)|^{2}&=|\D(\Bu)|^2-\frac{1}{2}\left| \p_r u^{\theta}-\frac{u^{\theta}}{r} \right|^2-\left| \frac{u^r}{r}\right|^2+\frac{1}{2}\left|\p_r u^{\theta}\right|^2,\\
            \left|\p_r u^{\theta}\right|&\le C\left( |\D(\Bu)|+\left|\frac{u^{\theta}}{r}\right|\right).
		\end{aligned}
		\ee
		Hence, by \eqref{Kororv} and \eqref{ineqnaDu}, one has
		\begin{equation}
			\begin{aligned}
				\| \nabla \Bu\|_{L^p(\mathcal{D}_R)} &\le C\left(\|\overline{\nabla}\Bu\|_{L^p(\mathcal{D}_R)}+ \left\|\frac{u^r}{r}\right\|_{L^p(\mathcal{D}_R)} +\left\|\frac{u^\theta}{r}\right\|_{L^p(\mathcal{D}_R)}      \right) \\
                &\le C\left(\|\overline{\D}(\Bu)\|_{L^p(\mathcal{D}_R)}+ \left\|\frac{u^r}{r}\right\|_{L^p(\mathcal{D}_R)} +\left\|\frac{u^\theta}{r}\right\|_{L^p(\mathcal{D}_R)}      \right) \\
                &\le C\left(\|\D(\Bu)\|_{L^p(\mathcal{D}_R)}+ \left\|\frac{u^r}{r}\right\|_{L^p(\mathcal{D}_R)} +\left\|\frac{u^\theta}{r}\right\|_{L^p(\mathcal{D}_R)} +\left\|\p_r u^{\theta}  \right\|_{L^p(\mathcal{D}_R)}   \right)\\
				 &\le C\left(\|\D(\Bu)\|_{L^p(\mathcal{D}_R)}+R^{-1}\|\Bu\|_{L^p(\mathcal{D}_R)}   \right) .
                \end{aligned}
		\end{equation}
		When  $R$ is sufficiently large, such that $CR^{-1}<1$, together with the Poincar\'e inequality and upon transforming the domain from $\mathcal{D}_R$ to $\OR$, one obtains
		\[
		\| \nabla \Bu\|_{L^p(\OR)} \le C \|\D(\Bu)\|_{L^p(\OR)} .
		\]

       Proof of $\text{Case}(2)$. Following the same verification as in $\text{Case}(1)$, one can confirm that Lemma \ref{Korn F} also holds for the domain $Q_1$ and thus Korn’s inequality \eqref{kk} is satisfied for $Q_1$. For general  vector field $\Bu$ in the domain $Q_R$, by employing the scaling transformation between $Q_1$ and $Q_R$, one can show that the constant $C$ in \eqref{ki2} is independent of $R$. Hence the proof of Lemma  \ref{korn} is completed. 
\end{proof}

	The Gagliardo-Nirenberg interpolation inequality in bounded domains is as follows.
	\begin{lemma}\label{G-Ninequality}(\hspace{1sp}\cite[Theorem 1.2]{LZ})
		Let $D$ be a bounded Lipschitz domain. Assume that $1 \le q,r \le + \infty$, $k,j \in \mathbb{N}$ with $j <k$, $\frac{j}{k}\le \theta\le 1$ and $p \in \mathbb{R}$ such that 
		\[
		\frac{1}{p}=\frac{j}{n}+\theta\left( \frac{1}{r}-\frac{k}{n}\right)+(1-\theta)\frac{1}{q}.
		\]
		Then there exists a constant C depending only on $n,k,q,r,\theta$ and $D$ such that for any $\Bu \in W^{k,r}(D)\cap L^q(D) $,
		\begin{equation}\label{GN}
			\begin{aligned}
				\|\nabla^j \Bu\|_{L^{p}(D)} \le C\|\nabla^k \Bu\|_{L^r(D)}^{\theta}\|\Bu\|_{L^q(D)}^{1-\theta}+C\|\Bu\|_{L^q(D)},
			\end{aligned}
		\end{equation}
		with the exception that if $1<r<+\infty $ and $k-j-\frac{n}{r} \in \mathbb{N}$, we must take $\frac{j}{k}\le \theta<1 $.
		
	\end{lemma}

	\begin{lemma}\label{SV}
		Let $\phi(t)$ be a nondecreasing function
		and $t_{0}>1$ be a fixed constant. Suppose that $\phi(t)$  is not identically zero.
		
		(a) For $2\leq p<3$, if $\phi(t)$ satisfies
		\begin{equation}\label{eqgrowth1}
			\phi(t)\leq  C_{1}	\phi^{\prime}(t)+C_{2}\left[\phi^{\prime}(t)\right]^{\frac{3}{p}} \quad \text{for  any} \ t\geq t_{0},
		\end{equation}
		then
		\begin{equation}\label{eqresult1}
			\varliminf_{t\rightarrow +\infty}t^{-\frac{3}{3-p}}\phi(t)>0.
		\end{equation}

		(b) For $2\leq p<3$, if $\phi(t)$ satisfies
		\begin{equation}\label{casebassump}
			\phi(t)\leq  C_{3}	\phi^{\prime}(t)+C_{4}t^{\frac{p-3}{p}}\left[\phi^{\prime}(t)\right]^{\frac{3}{p}} \quad \text{for  any} \ t\geq t_{0},
		\end{equation}
		then
		\begin{equation}\label{concluphit2}
			\varliminf_{t\rightarrow +\infty}t^{-\frac{6-p}{3-p}}\phi(t)>0.
		\end{equation}
		
		(c) For $2\leq p<3$, if $\phi(t)$ satisfies
		\begin{equation}\label{SVcasea2}
			\phi(t)\leq  C_5	t\phi^{\prime}(t)+C_6t^{\frac{3}{p}} \left[\phi^{\prime}(t) \right]^{\frac{3}{p}}\quad \text{for  any} \ t\geq t_{0},
		\end{equation}
		then
		\begin{equation}\label{SVc2}
			\varliminf_{t\rightarrow +\infty} (\ln t)^{-\frac{3}{3-p}}\phi(t)>0.
		\end{equation}
		
		(d) If $\phi(t)$ satisfies
		\begin{equation}\label{SVcased}
			\phi(t)\leq  C_7	\phi^{\prime}(t) \quad \text{for  any} \ t\geq t_{0},
		\end{equation}
		then 
		\begin{equation}\label{d2}
			\varliminf_{t\rightarrow +\infty}e^{-\frac{1}{C_7}t}\phi(t)>0.
		\end{equation}
		
		(e) If $\phi(t)$ satisfies
		\begin{equation}\label{SVcasea2cn}
			\phi(t)\leq  C_8 t	\phi^{\prime}(t) \quad \text{for  any} \ t\geq t_{0},
		\end{equation}
		then 
		\begin{equation}\label{}
			\varliminf_{t\rightarrow +\infty}t^{-\frac{1}{C_{8}}}\phi(t)>0.
		\end{equation}

	\end{lemma}
	\begin{proof} 
		{\bf Case (a)}\, Since $\phi(t)$ is not identically zero, we assume that there is a $t_1\geq t_0$ such that $\phi(t_1)> 0$. Due  to  \eqref{eqgrowth1}, it holds that 
		\begin{equation}
			C_{1}	\phi^{\prime}(t)+C_{2}\left[\phi^{\prime}(t)\right]^{\frac{3}{p}}\geq  \phi(t)\geq \phi(t_1)>0 \quad \text{for any}\ t\geq t_{1}.
		\end{equation}
		Hence there is a positive constant $\alpha_{1}$ such that 
		\begin{equation}\label{eqalpha1}
			\phi^{\prime}(t)>\alpha_{1} \quad \text{for any} \,\, t\geq t_{1}.
		\end{equation}
		Inserting \eqref{eqalpha1} into \eqref{eqgrowth1}, one has
		\begin{equation}
			\phi(t)\leq (C_{1}\alpha^{\frac{p-3}{p}}_{1}+C_{2})[\phi^{\prime}(t)]^{\frac{3}{p}}.
		\end{equation}
		Integrating both sides yields \eqref{eqresult1}.
		
		{\bf Case (b)}\, If $\phi(t)$ satisfies \eqref{casebassump}, it follows from $Case\,(a)$ that there is a $t_{2}>t_0$ and a constant $\alpha_{2}$ such  that
		\[
		C_{3}	\phi^{\prime}(t)+C_{4}t^{\frac{p-3}{p}}\left[\phi^{\prime}(t)\right]^{\frac{3}{p}}
		\geq \alpha_{2}t^{\frac{3}{3-p}} \quad \text{for any} \ t\geq t_{2}.
		\]
		Hence it holds that $C_{3}	\phi^{\prime}(t)\geq \frac{1}{2}\alpha_{2}t^{\frac{3}{3-p}}$ or
		$C_{4}t^{\frac{p-3}{p}}\left[\phi^{\prime}(t)\right]^{\frac{3}{p}}
		\geq \frac{1}{2}\alpha_{2}t^{\frac{3}{3-p}}$.
		Therefore, there exists a positive constant  $\alpha_{3}>0$ such that
		\[
		\phi^{\prime}(t)\geq \alpha_{3} t^{\frac{p^{2}-3p+9}{9-3p}} \quad \text{for any} \ t\geq t_{2}.
		\]
		Hence one derives
		\begin{equation}
			\phi(t)\leq C_{3}\alpha_{3}^{\frac{p-3}{p}} t^{\frac{p^{2}-3p+9}{9-3p}\frac{p-3}{p}}	\left[\phi^{\prime}(t)\right]^{\frac{3}{p}}+C_{4}t^{\frac{p-3}{p}}\left[\phi^{\prime}(t)\right]^{\frac{3}{p}}\leq Ct^{\frac{p-3}{p}}\left[\phi^{\prime}(t)\right]^{\frac{3}{p}},
		\end{equation}
		which implies \eqref{concluphit2}.

		{\bf Case (c)}\, Assume that $\phi(t_3)>0$, for some $t_3>t_0$,  then $\phi(t)$ satisfies 
       \begin{equation}
			 C_5	t\phi^{\prime}(t)+C_6t^{\frac{3}{p}} \left[\phi^{\prime}(t) \right]^{\frac{3}{p}}\ge \phi(t)\ge \phi(t_3)>0 \quad \text{for  any} \ t \geq t_{3}.
		\end{equation}
Hence it holds that $ C_5	t\phi^{\prime}(t)\ge \frac{1}{2}\phi(t_3)$ or $C_6t^{\frac{3}{p}} \left[\phi^{\prime}(t) \right]^{\frac{3}{p}}\ge \frac{1}{2}\phi(t_3)$, then there exists a positive constant $\alpha_{4}$ such that 
	\begin{equation}\label{eqalpha1n}
			t\phi^{\prime}(t)>\alpha_{4} \quad \text{for any} \,\, t\geq t_{3}.
		\end{equation}
Hence one obtains
\begin{equation}
  \phi(t)\le \left( C_5 \alpha_4^{\frac{p-3}{p}}+C_6\right)t^{\frac{3}{p}}[\phi'(t)]^{\frac{3}{p}},
\end{equation}
which implies \eqref{SVc2}.

		{\bf Case(d)}--{\bf Case(e)} can be proved in almost the same way as above, for which the details are omitted.
	\end{proof}

	\section{Liouville-type Theorem for Axisymmetric Solutions}\label{Sec3}
	In this section, we deal with the axisymmetric solution of the non-Newtonian fluid equations \eqref{eqsteadynNfs} in a slab with no-slip boundary conditions \eqref{noslipboun}. Before embarking on the proof of Theorem \ref{th:01}, we first prove a preparatory proposition.
	\begin{pro}\label{IG1}
		 Let $2\leq p<+\infty$. Assume that $\Bu$ be an axisymmetric weak solution of the equations \eqref{eqsteadynNfs} in a slab $\Omega=\mathbb{R}^{2}\times (0,1)$ with no-slip boundary conditions \eqref{noslipboun}. Let
		\begin{equation}\label{D-integral}
			\mathcal{E}(R)
			=\int_{0}^{1}\int_{\left\{x^{2}_{1}+x^{2}_{2}<R^{2}\right\}}(|\nabla\Bu|^{2}+|\D(\Bu)|^p)\,\dd x_{1}\dd x_{2}\dd x_{3}.
		\end{equation}
        
		$(b_1)$ For $2\leq p<3$, if $\Bu$ satisfies
		\begin{equation}\label{pes1}
			\varliminf_{R\rightarrow +\infty}R^{-\frac{6-p}{3-p}}\mathcal{E}(R)=0,
		\end{equation}
		then $\Bu \equiv 0$.
        
		$(b_2)$  For $p\ge 3$, there exists a constant $C_{\ast}(p) > 0$, such that if $\Bu$ satisfies 
		\begin{equation}\label{pes2}
			\varliminf_{R\rightarrow +\infty}e^{-\frac{R}{C_{\ast}(p)} }\mathcal{E}(R)=0,
		\end{equation}
        then $\Bu \equiv 0$.

	\end{pro}
	\begin{proof}
		First, one can define a cut-off function
		\be \label{cut-off}
		\varphi_R(r) = \left\{ \ba
		&1,\ \ \ \ \ \ \ \ \ \ r < R-1, \\
		&R-r,\ \ \ \ R-1 \leq r \leq R, \\
		&0, \ \ \ \ \  \ \ \ \ \ r > R.
		\ea  \right.
		\ee
		Multiplying the momentum equation  in \eqref{eqsteadynNfs}
		with $\Bu\varphi_{R}(r)$ and integrating by parts over the domain $\Omega$ yields
		\begin{equation}\label{EQF}
			\begin{split}
				&\int_{\Omega}(1+|\D(\Bu)|^2)^{\frac{p-2}{2}}|\D(\Bu)|^2\varphi_{R}\,\dd\Bx\\
				=&
				-\int_{\Omega}(1+|\D(\Bu)|^2)^{\frac{p-2}{2}}\D(\Bu):(\Bu\otimes\nabla\varphi_{R})\,\dd\Bx+\int_{\Omega}\frac{1}{2}|\Bu|^{2}\Bu\cdot
				\nabla\varphi_{R}\,\dd\Bx
				+\int_{\Omega}P\Bu\cdot\nabla\varphi_{R}\,\dd\Bx.
			\end{split}
		\end{equation}

    $Case (b_1)$. $2\leq p<3$. There exists some $C>0$, such that 
    \begin{equation}
        \int_{\Omega}(1+|\D(\Bu)|^2)^{\frac{p-2}{2}}|\D(\Bu)|^2\varphi_{R}\,\dd\Bx \ge C \int_{\Omega}(|\D(\Bu)|^2+|\D(\Bu)|^p)\varphi_{R}\,\dd\Bx.
    \end{equation}
    Moreover, one has 
\begin{equation}\label{Korneql}
    \begin{aligned}
        & \int_{\Omega}|\nabla \Bu|^2 \varphi_R\,\dd\Bx+\int_{\Omega} \nabla\varphi_R \cdot\nabla\Bu \cdot \Bu\,\dd\Bx \\
        =&\int_{\Omega} -\Delta \Bu \cdot \Bu \varphi_R\,\dd \Bx =  \int_{\Omega} -2\text{div} \D(\Bu) \cdot \Bu \varphi_R\,\dd \Bx            \\
        =&\int_{\Omega} 2|\D( \Bu)|^2 \varphi_R \,\dd\Bx+\int_{\Omega}2 \D(\Bu):(\Bu \otimes \nabla \varphi_R)\,\dd\Bx.
    \end{aligned}
\end{equation}
From \eqref{EQF}-\eqref{Korneql}, one obtains
		\begin{equation}\label{Es}
			\begin{split}
				C\int_{\Omega}(|\nabla\Bu|^{2}+|\D(\Bu)|^p)\varphi_{R}\,\dd\Bx &\le  \int_{\Omega}(1+|\D(\Bu)|^2)^{\frac{p-2}{2}}|\D(\Bu)|^2\varphi_{R}\,\dd\Bx\\
               &\quad +\left|\int_{\Omega}\D(\Bu):(\Bu \otimes \nabla \varphi_R)\,\dd\Bx\right|+\left|\int_{\Omega} \nabla\varphi_R \cdot \nabla\Bu \cdot \Bu\,\dd\Bx \right|\\
				&\leq \left|\int_{\Omega}(1+|\D(\Bu)|^2)^{\frac{p-2}{2}}\D(\Bu):(\Bu\otimes\nabla\varphi_{R})\,\dd\Bx\right|\\
                &\quad +\left|\int_{\Omega}\frac{1}{2}|\Bu|^{2}\Bu\cdot
				\nabla\varphi_{R}\,\dd\Bx\right|
                +\left|\int_{\Omega}P\Bu\cdot\nabla\varphi_{R}\,\dd\Bx\right|\\
                &\quad +\left|  \int_{\Omega}\D(\Bu):(\Bu \otimes \nabla \varphi_R)\,\dd\Bx\right|+\left|\int_{\Omega} \nabla\varphi_R \cdot \nabla\Bu \cdot \Bu\,\dd\Bx  \right|.
			\end{split}
		\end{equation} 
 By H{\"o}lder's inequality, Korn's inequality \eqref{yski} and Poincar\'e  inequality, one obtains
		\begin{equation}\label{ess1}
			\begin{aligned}
				& \left|\int_{\Omega}(1+|\D(\Bu)|^2)^{\frac{p-2}{2}}\D(\Bu):(\Bu\otimes\nabla\varphi_{R})\,\dd\Bx\right|\\
				\leq &\,C
				\left|\int_{\OR}(|\D(\Bu)|+|\D(\Bu)|^{p-1}):|(\Bu\otimes\nabla\varphi_{R})|\,\dd\Bx\right|\\
 \le &\,C \left(\|\D(\Bu)\|_{L^2(\OR)}\cdot \|\Bu\|_{L^2(\OR)}+\|\D(\Bu)\|_{L^p(\OR)}^{p-1}\cdot \|\Bu\|_{L^p(\OR)} \right)\\
 \le & \,C\left(\| \nabla \Bu\|_{L^2(\OR)}^2+ \|\D(\Bu)\|_{L^p(\OR)}^{p-1} \|\nabla \Bu\|_{L^p(\OR)}\right)\\
 \le & \,C\left(\| \nabla \Bu\|_{L^2(\OR)}^2+ \|\D(\Bu)\|_{L^p(\OR)}^p\right)
	\end{aligned}
		\end{equation}
and 
        \begin{equation}\label{fj}
            \begin{aligned}
            &\left|  \int_{\Omega}\D(\Bu):(\Bu \otimes \nabla \varphi_R)\,\dd\Bx\right|+\left|\int_{\Omega}\nabla\varphi_R \cdot  \nabla\Bu \cdot \Bu\,\dd\Bx  \right|\\
        \le &\,C\left( \|\D(\Bu)\|_{L^2(\OR)}+\|\nabla \Bu\|_{L^2(\OR)}\right)\cdot \|\Bu\|_{L^2(\OR)}
            \le C\|\nabla \Bu\|_{L^2(\OR)}^2.
            \end{aligned}
        \end{equation}
By Gagliardo-Nirenberg interpolation inequality \eqref{GN}, Korn's inequality \eqref{yski} and Poincar\'e inequality, one has
		\begin{equation}\label{ess2}
			\begin{split}
				&\left|\int_{\Omega}\frac{1}{2}|\Bu|^{2}\Bu\cdot
				\nabla\varphi_{R}\,\dd\Bx\right|\\
                 \le &\, CR \left|\int_{0}^{1}\int_{R-1}^{R}|\Bu|^{3}\,\dd r\dd z\right|
				\le CR\left(\|\Bu\|_{L^p(D_R)}^{5-\frac{6}{p}}\|(\p_r,\p_z) \Bu\|_{L^p(D_R)}^{\frac{6}{p}-2}+\|\Bu\|^3_{L^p (D_R)}\right)\\
				 \le& \,CR^{1-\frac{3}{p}} \left(\|\Bu\|_{L^p(\OR)}^{5- \frac{6}{p}} \|\nabla \Bu\|_{L^p(\OR)}^{\frac{6}{p}-2}+ \|\nabla \Bu\|_{L^p(\OR)}^3 \right)
				\le CR^{1-\frac{3}{p}} \| \D(\Bu)\|_{L^p(\OR)}^3.
			\end{split}
		\end{equation}

		It remains to estimate 
		\begin{equation}\label{espress}
			\left|\int_{\Omega}P\Bu\cdot\nabla\varphi_{R}\,\dd\Bx\right|.
		\end{equation}
	The divergence free condition for the axisymmetric solution reduces to 
		\[
		\partial_r(ru^{r})+\partial_z(ru^{z})=0,
		\]
	which combining with the no-slip boundary conditions \eqref{noslipboun} gives 
		\[
		\partial_r \int_{0}^{1}ru^{r}\,\dd z=-\int_{0}^{1}\partial_z(ru^{z})\,\dd z=0.
		\]
	 Consequently,  it holds that
		\begin{equation}
			\int_{0}^{1}ru^{r}\,\dd z=0 \quad  \text{and} \quad
			\int_{0}^{1}\int_{R-1}^{R}ru^{r}\,\dd r\dd z=0.
		\end{equation}
		By Lemma \ref{Bogovskii}, there is a vector-valued function $\widetilde{\BPsi}_{R}(r, z) \in W^{1, p}_{0}(D_R; \mathbb{R}^2)$, such that 
		\begin{equation}\label{translate}
		\partial_r 	\widetilde{\Psi}_{R}^r  + \partial_z	\widetilde{\Psi}_{R}^z=r u^r,\ \ \mbox{in}\ D_R 
		\end{equation}
		and
			\begin{equation}\label{syB}
			\| (\p_r,\p_z )\widetilde{\BPsi}_{R}\|_{L^p(D_R)}\le C\|ru^r \|_{L^p(D_R)} \le CR^{1-\frac{1}{p}}\|u^r\|_{L^p(\OR)}. 
	\end{equation}
	One notes that the constant $C$ in \eqref{syB} is independent of $R$. 
	Let 
		\begin{equation}\label{translate1}
		\Psi^r_R = \frac{\widetilde{\Psi}_{R}^r}{r}, \quad \quad \Psi^z_R=\frac{ \widetilde{\Psi}_{R}^z}{r}, \ \ \ \BPsi_R= \Psi_R^r \Be_r + \Psi_R^z \Be_z. 
	\end{equation}
	Now	\eqref{translate} 	can be written as 	
	\begin{equation}\label{divergence}
	{\rm div}\BPsi_R = 	\frac{1}{r} \left[\partial_r (r \Psi^r_R)+ \partial_z (r\Psi_R^z) \right] = u^r, \,\,\,\,\text{in} \,\,\,\OR. 
		\end{equation}
	Moreover, by Poincar\'e inequality, \eqref{ineqnaDu} and \eqref{syB},  it holds that 
	\begin{equation}\begin{aligned}		\label{esbm}	
		\|\nabla \BPsi_R\|_{L^p(\OR)} & \leq \|\partial_r(\Psi_R^r, \Psi_R^z)\|_{L^p(\OR)} + \|\partial_z(\Psi_R^r, \Psi_R^z)\|_{L^p(\OR)} + \left\|\frac{\Psi_R^r}{r}\right\|_{L^p(\OR)} \\
		& \leq C R^{\frac1p}\|\partial_r (\Psi_R^r, \Psi_R^z)\|_{L^p(D_R)} +C R^{\frac1p}\|\partial_z (\Psi_R^r, \Psi_R^z)\|_{L^p(D_R)} \\
		&\leq CR^{\frac1p-1}\|\partial_r (\widetilde{\Psi}_R^r, \widetilde{\Psi}_R^z)\|_{L^p(D_R)} +  CR^{\frac1p-1}\|\partial_z (\widetilde{\Psi}_R^r, \widetilde{\Psi}_R^z)\|_{L^p(D_R)}\\
		&\leq CR^{\frac1p-1}\|ru^r\|_{L^p(D_R)} \leq C\|u^r\|_{L^p(\OR)}. 
\end{aligned}	\end{equation}	
	Therefore, combining \eqref{esbm}, Poincar\'e inequality, Korn's type inequality \eqref{yski} and H{\"o}lder's inequality, one obtains 
				
		\begin{equation}\label{esp}
			\begin{split}
				\left|\int_{\Omega}P\Bu\cdot\nabla\varphi_{R}\,\dd\Bx\right| &= \left|  \int_{\OR} P u^r\,\dd\Bx \right|= \left|  \int_{\OR} P \text{div} \BPsi_{R}\,\dd\Bx \right|= \left|  \int_{\OR} -\nabla P \cdot \BPsi_{R}\,\dd\Bx \right| \\
				&= \left|  \int_{\OR}     [-\text{div} A_p(\Bu)+(\Bu \cdot \nabla )\Bu ]   \cdot \BPsi_{R}        \,\dd\Bx \right|\\
				&= 	 \left|\int_{\OR}(1+|\D(\Bu)|^2)^{\frac{p-2}{2}}\D(\Bu):\D\BPsi_{R} \,\dd\Bx-\int_{\OR}(\Bu \cdot \nabla )\BPsi_{R} \cdot \Bu\,\dd\Bx \right| \\
				&\le C \left(\|\D(\Bu)\|_{L^2(\OR)}\cdot \|\nabla \BPsi_{R}\|_{L^2(\OR)}+\|\D(\Bu)\|_{L^p(\OR)}^{p-1}\cdot \|\nabla \BPsi_{R}\|_{L^p(\OR)} \right)\\   
				&\quad + C\|\nabla \BPsi_R \|_{L^p(\OR)}\cdot \|\Bu\|^2_{L^{\frac{2p}{p-1}}(\OR)}\\
                &\le  C\left(\|\nabla \Bu\|_{L^2(\OR)}^2+\|\D(\Bu)\|_{L^p(\OR)}^p  \right) + C\|u^r\|_{L^p(\OR)}\cdot \|\Bu\|^2_{L^{\frac{2p}{p-1}}(\OR)}\\
				&\le  C\left(\|\nabla \Bu\|_{L^2(\OR)}^2+\|\D(\Bu)\|_{L^p(\OR)}^p  \right)+ CR^{1-\frac{3}{p}}\| \D(\Bu)\|_{L^p(\OR)}^3,\\
			\end{split}
		\end{equation}
		here the last inequality is due to 
		\begin{equation}\label{PL}
			\begin{aligned}
				\|\Bu\|_{L^q(\OR)}^2 & =\left(2\pi \int_0^1\int_{R-1}^R|\Bu|^q  r\,\dd r\dd z \right)^{\frac{2}{q}} \le CR^{\frac{2}{q}}\|\Bu\|_{L^q(D_R)}^2
			 	\le CR^{\frac{2}{q}}\|(\p_r,\p_z)\Bu\|_{L^{\frac{2q}{q+2}}(D_R)}^2  \\
				& \le CR^{\frac{2}{q}} \|\nabla \Bu \|_{L^p(D_R)}^2  \le   C R^{1-\frac{3}{p}}\|\nabla\Bu\|_{L^p(\OR)}^2 \\
					& \le C R^{1-\frac{3}{p}}\|\D(\Bu)\|_{L^p(\OR)}^2,
			\end{aligned}    
		\end{equation}
		where $q=\frac{2p}{p-1}$. 
		Combining \eqref{Es}--\eqref{ess2} and  \eqref{esp}, one arrives at
		\begin{equation}\label{eqfinaq1}
			\int_{\Omega}(|\nabla\Bu|^{2}+|\D(\Bu)|^p)\varphi_{R}\,\dd\Bx \le  C \left(\|\nabla \Bu\|_{L^2(\OR)}^2+\|\D(\Bu)\|_{L^p(\OR)}^p  \right)+ CR^{1-\frac{3}{p}}\| \D(\Bu)\|_{L^p(\OR)}^3.
		\end{equation}

		Let 
		\begin{equation}\label{YR}
			Y(R)=
			\int_{0}^{1}\iint_{\mathbb{R}^{2}}
			(|\nabla\Bu|^{2}+|\D(\Bu)|^p)\varphi_{R}\left(\sqrt{x_{1}^{2}+x_{2}^{2}}\right)
			\, \dd x_{1} \dd x_{2} \dd x_{3}.
		\end{equation}
	It can be verified that 
		\[
		Y^{\prime}(R)=\int_{\OR}|\nabla\Bu|^{2}+|\D(\Bu)|^p \,\dd\Bx.
		\]
		Hence the estimate \eqref{eqfinaq1} can be written as 
		\begin{equation}\label{Casea1es}
			Y(R)\le C_1 Y'(R)+C_2 R^{1-\frac{3}{p}}[Y'(R)]^{\frac{3}{p}}.  
		\end{equation}
		Then $Case(b_1)$ of Proposition \ref{IG1} follows from Lemma \ref{SV}.
		
		$Case(b_2)$. $p\ge 3$. The main idea of the proof is the same as that for $Case(b_1)$. The inequality \eqref{Es} and the estimate \eqref{ess1}-\eqref{fj} still hold. By Poincar\'e inequality and Korn's inequality \eqref{yski}, one has 
		\begin{equation}\label{estimate1-1}
			\begin{split}
				\left|\int_{\Omega}\frac{1}{2}|\Bu|^{2}\Bu\cdot
				\nabla\varphi_{R}\,\dd\Bx\right| & \le CR \left|\int_{0}^{1}\int_{R-1}^{R}|\Bu|^{3}\,\dd r\dd z\right|
				\le CR\left(\int_{0}^{1}\int_{R-1}^{R}|\Bu|^{p}\,\dd r\dd z\right)^{\frac{3}{p}}\\
				& \le CR^{1- \frac3p}\|\nabla \Bu\|_{L^p(\OR)}^3 \le CR^{1-\frac{3}{p}} \| \D(\Bu)\|_{L^p(\OR)}^3.
			\end{split}
		\end{equation}
		Or one has another estimate 
		\begin{equation}\label{estimate1-2}
			\begin{split}
				\left|\int_{\Omega}\frac{1}{2}|\Bu|^{2}\Bu\cdot
				\nabla\varphi_{R}\,\dd\Bx\right|\le CR \left|\int_{0}^{1}\int_{R-1}^{R}|\Bu|^{3}\,\dd r\dd z\right|&\le CR \left(\int_{0}^{1}\int_{R-1}^{R}|(\partial_{r},\partial_{z})\Bu|^{2}\,\dd r\dd z\right)^{\frac{3}{2}}\\
				&\le CR^{-\frac{1}{2}}\|\nabla \Bu\|_{L^2(\OR)}^3.
			\end{split}
		\end{equation}
		The estimate for the pressure term is analogous to that for \eqref{esp}. Let $ \BPsi_R$ be the functions defined by \eqref{translate1}. By Poincar\'e inequality and Korn's inequality \eqref{yski}, one obtains
		\begin{equation}\label{estimate2-1}
			\begin{split}
				\left|\int_{\Omega}P\Bu\cdot\nabla\varphi_{R}\,\dd\Bx\right|
				&= 	 \left|\int_{\OR}(1+|\D(\Bu)|^2)^{\frac{p-2}{2}}\D(\Bu):\D\BPsi_{R} \,\dd\Bx-\int_{\OR}(\Bu \cdot \nabla )\BPsi_{R} \cdot \Bu\,\dd\Bx \right| \\
				&\le C \left(\|\D(\Bu)\|_{L^2(\OR)}\cdot \|\nabla \BPsi_{R}\|_{L^2(\OR)}+\|\D(\Bu)\|_{L^p(\OR)}^{p-1}\cdot \|\nabla \BPsi_{R}\|_{L^p(\OR)} \right)\\   
				&\quad + C\|\nabla \BPsi_{R} \|_{L^p(\OR)}\cdot \|\Bu\|^2_{L^{\frac{2p}{p-1}}(\OR)}\\
				& \le C \left( \|\D(\Bu)\|_{L^2(\OR)} \|u^r\|_{L^2(\OR)} +  \|\D(\Bu)\|_{L^p(\OR)}^{p-1} \|u^r\|_{L^p(\OR)} \right) \\ 
				& \quad + C\|\nabla \BPsi_{R} \|_{L^p(\OR)}\cdot \|\Bu\|^2_{L^{\frac{2p}{p-1}}(\OR)} \\
				&\le  C\left(\|\nabla \Bu\|_{L^2(\OR)}^2+\|\D(\Bu)\|_{L^p(\OR)}^p  \right)+ CR^{1-\frac{3}{p}}\| \D(\Bu)\|_{L^p(\OR)}^3,\\
			\end{split}
		\end{equation}
		where the last inequality is due to 
		\begin{equation}\label{PL1}
			\begin{aligned}
				\|\Bu\|_{L^{\frac{2p}{p-1}}(\OR)}^2&=\left(2\pi \int_0^1\int_{R-1}^R|\Bu|^{\frac{2p}{p-1}} r\,drdz \right)^{\frac{p-1}{p}}\\
                &\le CR^{1-\frac{1}{p}}\|\Bu\|_{L^{\frac{2p}{p-1}}(D_R)}^2
				\le CR^{1-\frac{1}{p}}\|\Bu\|_{L^p(D_R)}^2\\
				& \le CR^{1-\frac{3}{p}}\|\Bu\|_{L^p(\OR)}^2 \le CR^{1-\frac{3}{p}}\|\nabla \Bu\|_{L^p(\OR)}^2
				\\
                &\le C R^{1-\frac{3}{p}}\|\D(\Bu)\|_{L^p(\OR)}^2.
			\end{aligned}    
		\end{equation}
		Or
		\begin{equation}\label{estimate2-2}
			\begin{split}
				\left|\int_{\Omega}P\Bu\cdot\nabla\varphi_{R}\,\dd\Bx\right| 
                \le 
                C\left(\|\nabla \Bu\|_{L^2(\OR)}^2+\|\D(\Bu)\|_{L^p(\OR)}^p  \right)+ CR^{-\frac{1}{2}}\| \nabla\Bu\|_{L^2(\OR)}^3.
			\end{split}
		\end{equation}
Combining the estimates \eqref{Es}, \eqref{ess1}, \eqref{estimate1-1} and \eqref{estimate2-1}, one arrives at
		\begin{equation}\label{Casea2es}
			Y(R)\le C_3 Y'(R)+C_4 R^{1-\frac{3}{p}}[Y'(R)]^{\frac{3}{p}}.
		\end{equation}
		Or, combining the estimates \eqref{Es}, \eqref{ess1}, \eqref{estimate1-2} and \eqref{estimate2-2}, one obtains
		\begin{equation}\label{Casea22es}
			Y(R)\le C_5 Y'(R)+C_6 R^{-\frac{1}{2}}[Y'(R)]^{\frac{3}{2}}.  
		\end{equation}

		According to the estimate \eqref{Casea22es} and Case $(b)$ of Lemma \ref{SV} with $p=2$, if $Y(R)$ is not identically zero, there exist $C_0$ and $R_0$ such that 
		\[
		Y(R) \ge C_0 R^4,      \quad  \text{for any} \, \, R\ge R_0. 
		\]
		It implies that 
		\begin{equation}\label{R4}
		 C_5 Y'(R)+C_6 R^{-\frac{1}{2}}[Y'(R)]^{\frac{3}{2}}  \ge C_0 R^4, \quad \text{for any} \, \, R\ge R_0.
		\end{equation}
	Hence, there exists some constant $C_7$ such that
		\begin{equation}\label{R4_1}
			Y'(R) \ge C_7 R^{\frac{8}{3}}, \quad \text{for any} \, \, R\geq R_0. 
		\end{equation}
	Then taking \eqref{R4_1} into  \eqref{Casea2es}, one arrives at
		\begin{equation}
			Y(R)\le C_3 Y'(R)+C_4 R^{1-\frac{3}{p}}[Y'(R)]^{\frac{3}{p}} \le C_3 Y'(R)+ C_4 R^{1-\frac{3}{p}}R^{\frac{8}{3}\cdot \frac{3-p}{p}} Y'(R).  
		\end{equation}
		Since $p \ge 3$ and $1-\frac{3}{p}+\frac{8(3-p)}{3p}\le 0$, there exists $R_0 $ sufficiently large, such that $C_4 R^{1-\frac{3}{p}}R^{\frac{8}{3}\cdot \frac{3-p}{p}}\le C_4$, as $ R>R_0$. Hence the estimate \eqref{Casea2es} can be written as 
		\begin{equation}\label{CP3}
			Y(R) \le C_8 Y'(R), 
		\end{equation}
		where $C_8=C_3+C_4$. And then $Case(b_2)$ of Proposition \ref{IG1} follows from Lemma \ref{SV}.
	\end{proof}
	
	Now we are ready for the proof of Theorem \ref{th:01}.
	\begin{proof}[Proof of Theorem \ref{th:01}]
		$Case(a_1)$. $2\le p<3$. We need to estimate \eqref{ess1}-\eqref{ess2} and \eqref{esp} in a different way.
		Instead of \eqref{ess1} and \eqref{fj}, one has 
		\begin{equation}\label{ees1}
			\begin{aligned}
				&\left|\int_{\Omega}(1+|\D(\Bu)|^2)^{\frac{p-2}{2}}\D(\Bu):(\Bu\otimes\nabla\varphi_{R})\,\dd\Bx\right|\\
				\leq&\,
				C\int_{\OR}(|\D(\Bu)|+|\D(\Bu)|^{p-1}): \left|(\Bu \otimes \nabla \varphi_R)\right|\,\dd\Bx  \\
			\le &\,C \left( \|\nabla \Bu\|_{L^2(\OR)}\cdot \|\Bu\|_{L^2(\OR)} +\|\D(\Bu)\|_{L^p(\OR)}^{p-1} \cdot \|\Bu\|_{L^p(\OR)} \right)\\
			 \le& \,C\left( R^{\frac{1}{2}}\|\Bu\|_{L^{\infty}(\OR)}\|\nabla \Bu \|_{L^2{(\OR)}} +R^{\frac{1}{p}}\|\Bu\|_{L^{\infty}(\OR)}\|\D(\Bu)\|_{L^p(\OR)}^{p-1}   \right)
			\end{aligned}
		\end{equation}		
and 
\begin{equation}\label{fj1}
    \begin{aligned}
       &\left|  \int_{\Omega}\D(\Bu):(\Bu \otimes \nabla \varphi_R)\,\dd\Bx\right|+\left|\int_{\Omega} \nabla\varphi_R \cdot\nabla\Bu \cdot  \Bu\,\dd\Bx  \right|\\
        \le& \,C\left(\|\D(\Bu)\|_{L^2(\OR)}+\|\nabla \Bu\|_{L^2(\OR)}\right)\cdot \|\Bu\|_{L^2(\OR)}\\
        \le&\, CR^{\frac{1}{2}}\|\Bu\|_{L^{\infty}(\OR)}\|\nabla \Bu \|_{L^2{(\OR)}}.
    \end{aligned}
\end{equation}
		Using Poincar\'e inequality  and Korn's inequality \eqref{yski} yields
		\begin{equation}\label{eess2}
			\begin{aligned}
				\left| \int_{\Omega} \frac{1}{2} |\Bu|^2 \Bu \cdot \nabla \varphi_R \,\dd\Bx\right|  &\le CR\|\Bu\|^{4-p}_{L^{\infty}(\OR)}\int_{0}^{1} \int_{R-1}^{R} |(u^r,u^\theta,u^z)|^{p-1}\,\dd r \,\dd z \\
				& \le CR\cdot R^{\frac{1}{p}-1}\|\Bu\|^{4-p}_{L^\infty(\OR)} \|\nabla \Bu \|_{L^p(\OR)}^{p-1}\\
				&\le  CR^{\frac{1}{p}}\|\Bu\|_{L^{\infty}(\OR)}^{4-p}\|\D(\Bu)\|_{L^p(\OR)}^{p-1}.
			\end{aligned}
		\end{equation}
		Compared with \eqref{esp}, after making a subtle refinement to the estimate for the pressure term, one obtains
		\begin{equation}\label{espress1}
			\begin{split}
				\left|\int_{\Omega}P\Bu\cdot\nabla\varphi_{R}\,\dd\Bx\right| 
				&= 	 \left|\int_{\OR}(1+|\D(\Bu)|^2)^{\frac{p-2}{2}}\D(\Bu):\D\BPsi_{R} \,\dd\Bx-\int_{\OR}(\Bu \cdot \nabla )\BPsi_{R} \cdot \Bu\,\dd\Bx \right| \\
				&\leq
				C\int_{\OR}(|\D(\Bu)|+|\D(\Bu)|^{p-1}): \left|\D\BPsi_{R}\right|\,\dd\Bx \\
				&\quad + C\|\Bu\|_{L^{\infty}(\OR)}^{3-p}\|\nabla \BPsi_{R} \|_{L^p(\OR)}\cdot \|\Bu\|^{p-1}_{L^p(\OR)} \\
				& \le C\left( R^{\frac{1}{2}}\|\Bu\|_{L^{\infty}(\OR)}\|\nabla \Bu \|_{L^2{(\OR)}} +R^{\frac{1}{p}}\|\Bu\|_{L^{\infty}(\OR)}\|\D(\Bu)\|_{L^p(\OR)}^{p-1}   \right)\\
				&\quad +CR^\frac{1}{p} \|\Bu\|_{L^{\infty}(\OR)}^{4-p}\|\D(\Bu)\|_{L^p(\OR)}^{p-1}.
			\end{split}
		\end{equation}
	The above computations imply
		\begin{equation}\label{esssssss1}
			Y(R) \le  C_9 R^{\frac{1}{2}}\|\Bu\|_{L^{\infty}(\OR)}[Y'(R)]^{\frac{1}{2}}+C_{10}R^{\frac{1}{p}}\left(\|\Bu\|_{L^{\infty}(\OR)}+\|\Bu\|_{L^{\infty}(\OR)}^{4-p}\right)[Y'(R)]^{1-\frac{1}{p}}, 
		\end{equation}
		where $Y(R)$ is defined in \eqref{YR}.
	
	Suppose $\Bu$ is not identically zero, according to Proposition \ref{IG1}, it holds that 
	\begin{equation}\label{suppose-1}
		\varliminf_{R\rightarrow +\infty}R^{-\frac{6-p}{3-p}}Y(R)>0.
	\end{equation}
	Equivalently, there exists some $R_0>1$, $C_0>0$, such that 
		\begin{equation}\label{caseb1}
				R^{-\frac{6-p}{3-p}}Y(R) \ge C_0,   \quad \text{for any} \,\, R\ge R_0.
			\end{equation}	
	And taking \eqref{caseb1} into \eqref{Casea1es}, one has
	\begin{equation}
				 C_1 Y'(R)+ C_2 R^{1-\frac{3}{p}}[Y'(R)]^{\frac{3}{p}}\ge Y(R) \ge C_0 R^{\frac{6-p}{3-p}},\quad \text{for any} \, \, R\geq R_0. 
			\end{equation}
	Hence, it holds that
		\be\label{caseb1esp}
	Y'(R) \ge CR^{\frac{3}{3-p}}.
	\ee
	From \eqref{esssssss1}, one arrives at
		\be\label{casebs1}
		\begin{aligned}
			Y(R) &\le  C_9 R^{\frac{1}{2}}\|\Bu\|_{L^{\infty}(\OR)}[Y'(R)]^{\frac{1}{p}-\frac{1}{2}+\frac{p-1}{p}}+C_{10}R^{\frac{1}{p}}\left(\|\Bu\|_{L^{\infty}(\OR)}+\|\Bu\|_{L^{\infty}(\OR)}^{4-p}\right)[Y'(R)]^{\frac{p-1}{p}} \\
			&\le C_{11}R^{\frac{1}{p}}\left(\|\Bu\|_{L^{\infty}(\OR)}+\|\Bu\|_{L^{\infty}(\OR)}^{4-p}\right)[Y'(R)]^{\frac{p-1}{p}},
		\end{aligned}
		\ee
	where in the last inequality we have used  \eqref{caseb1esp} to obtain
		\[
		R^{\frac{1}{2}-\frac{1}{p}}[Y'(R)]^{\frac{1}{p}-\frac{1}{2}} \le CR^{\frac{2-p}{2(3-p)}} \le C , \quad \text{for any} \, \, R\geq R_0.
		\]
	
Note that $\Bu$ satisfies \eqref{LIes1}, i.e., for any small $\epsilon$>0, there exists a $R_0(\epsilon)>2$ such that 
		\be\label{peeeeeeess1}
		\|\Bu\|_{L^{\infty}(\OR)} \le \epsilon R^{\frac{1}{3-p}}, \quad  \text{for any} \, R \ge R_0(\epsilon).
		\ee
	Taking the assumption \eqref{peeeeeeess1} into \eqref{casebs1}, one has 
		\begin{equation}\label{escaseb1l}
			\frac{Y'(R)}{[Y(R)]^{\frac{p}{p-1}}} \ge [C_{11}\epsilon^{(4-p)}]^{-\frac{p}{p-1}}R^{-(\frac{1}{p}+\frac{4-p}{3-p})\cdot \frac{p}{p-1}}. 
		\end{equation}
		If $\Bu$ is not zero, according to $Case (b_1)$ of Proposition \ref{IG1}, $Y(R)$ must be unbounded as $R \rightarrow +\infty$. For every $R$ sufficiently large, integrating \eqref{escaseb1l} over $[R,+\infty)$ one arrives at 
		\be\label{ppeess1n} 
		Y(R)\cdot R^{-\frac{6-p}{3-p}} \le (C_{11}\epsilon^{(4-p)})^p.
		\ee
		Since $\epsilon$ can be arbitrarily small, this leads to a contradiction with \eqref{caseb1}. Hence the proof for $Case (a_1)$ is completed.

		$Case (a_2)$. $p\ge 3$. Estimates \eqref{ees1} and \eqref{fj1} are similar to $Case(a_1)$, it is necessary to make a slight adjustment to the estimates  \eqref{eess2} and \eqref{espress1}. Using Poincar\'e inequality yields
		\begin{equation}\label{eesb2}
			\begin{aligned}
				\left| \int_{\Omega} \frac{1}{2} |\Bu|^2 \Bu \cdot \nabla \varphi_R\, \dd \Bx\right|  &\le CR\|\Bu\|^{2}_{L^{\infty}(\OR)}\int_{0}^{1} \int_{R-1}^{R} |(u^r,u^\theta,u^z)|\,\dd r \,\dd z \\
				& \le CR^{\frac{1}{2}}\|\Bu\|^{2}_{L^\infty(\OR)} \|\nabla \Bu \|_{L^2(\OR)}.\\
			\end{aligned}
		\end{equation}
		The similar computations as \eqref{espress1} show that 
		\begin{equation}\label{espressb2}
			\begin{split}
				\left|\int_{\Omega}P\Bu\cdot\nabla\varphi_{R}\,\dd\Bx\right| 
				&= 	 \left|\int_{\OR}(1+|\D(\Bu)|^2)^{\frac{p-2}{2}}\D(\Bu):\D\BPsi_{R} \,\dd\Bx-\int_{\OR}(\Bu \cdot \nabla )\BPsi_{R} \cdot \Bu\,\dd\Bx \right| \\
				&\leq
				C\int_{\OR}(|\D(\Bu)|+|\D(\Bu)|^{p-1}): \left|\D\BPsi_{R} \right|\,\dd\Bx  \\
				&\quad + C\|\nabla \BPsi_{R} \|_{L^2(\OR)}\cdot \|\Bu\|^{2}_{L^4(\OR)} \\
				& \le C\left( R^{\frac{1}{2}}\|\Bu\|_{L^{\infty}(\OR)}\|\nabla \Bu \|_{L^2{(\OR)}} +R^{\frac{1}{p}}\|\Bu\|_{L^{\infty}(\OR)}\|\D(\Bu)\|_{L^p(\OR)}^{p-1}   \right)\\
				&\quad +CR^\frac{1}{2} \|\Bu\|_{L^{\infty}(\OR)}^{2}\|\nabla \Bu\|_{L^2(\OR)}.
			\end{split}
		\end{equation}
		Combining the estimates  \eqref{ees1}-\eqref{fj1} and \eqref{eesb2}-\eqref{espressb2}, one derives
		\begin{equation}\label{esb2es}
			Y(R) \le  C_{12} R^{\frac{1}{2}}\left(\|\Bu\|_{L^{\infty}(\OR)}+\|\Bu\|_{L^{\infty}(\OR)}^{2}\right)[Y'(R)]^{\frac{1}{2}}+C_{13}R^{\frac{1}{p}}\|\Bu\|_{L^{\infty}(\OR)}[Y'(R)]^{1-\frac{1}{p}}. 
		\end{equation}
		Similarly, according to $Case (b_2)$ of Proposition \ref{IG1}, there exists a constant $C_{\ast}(p)=C_8$, such that if $\Bu$ satisfies
		\[
		\varliminf_{R\rightarrow +\infty}e^{-\frac{R}{C_8} }Y(R)=0,
		\]
		then $\Bu \equiv0$, and thus $Case (a_2)$ holds.\\
		Otherwise, if 
		\[
		\varliminf_{R\rightarrow +\infty}e^{-\frac{R}{C_8}}Y(R)\ne0,
		\]
		then there exist  $C_0$ and $R_0>0$ such that
		\begin{equation}\label{caseb2}
			e^{-\frac{R}{C_8}}Y(R) \ge C_0,   \quad  \text{for any } \, R\ge R_0.
		\end{equation}
		Combining \eqref{Casea2es} and \eqref{caseb2}, one obtains
		\begin{equation}\label{1}
			\begin{aligned}
				C_3Y'(R) \ge \frac{1}{2}C_0 e^{ \frac{R}{C_8}} \quad \text{or } \quad C_4 R^{\frac{p-3}{p}}[Y'(R)]^{\frac{3}{p}}\ge \frac{1}{2}C_0 e^{ \frac{R}{C_8}}.
			\end{aligned} 
		\end{equation}
	Thus, it holds that
		\be\label{caseb2es}
		Y'(R) \ge \frac{C_0}{2C_3}e^{\frac{R}{C_8}}.
		\ee
	Hence  $Y(R)$ in the estimate \eqref{esb2es} can be reduced to 
		\be\label{casebs1n}
		\begin{aligned}
			Y(R) &\le  C_{12} R^{\frac{1}{2}}\left(\|\Bu\|_{L^{\infty}(\OR)}+\|\Bu\|_{L^{\infty}(\OR)}^{2}\right)[Y'(R)]^{\frac{1}{p}-\frac{1}{2}+\frac{p-1}{p}}+C_{13}R^{\frac{1}{p}}\|\Bu\|_{L^{\infty}(\OR)}[Y'(R)]^{\frac{p-1}{p}} \\
			&\le C_{14}R^{\frac{1}{p}}\left(\|\Bu\|_{L^{\infty}(\OR)}+\|\Bu\|_{L^{\infty}(\OR)}^{2}\right)[Y'(R)]^{\frac{p-1}{p}},
		\end{aligned}
		\ee
		where the last inequality is due to
		\[
		R^{\frac{1}{2}-\frac{1}{p}}[Y'(R)]^{\frac{1}{p}-\frac{1}{2}} \le \left(\frac{C_0}{2C_{3}}\right)^{\frac{2-p}{2p}}R^{\frac{1}{2}-\frac{1}{p}}\cdot e^{\frac{R}{C_8}\cdot\frac{2-p}{2p}} \le C,   \quad \text{for any}\,\, p\ge 3.
		\]

		Suppose $\Bu$ is not identically equal to zero and $\Bu$ satisfies \eqref{LIes2}. For any small $\epsilon$>0, there exist some constants $R_0(\epsilon)>2$ and $C(p)=3C_8p>0$ such that 
		\be\label{esp>3}
		\|\Bu\|_{L^{\infty}(\OR)} \le \epsilon e^{\frac{ R}{3C_8 p}}, \quad  \text{for any } \, R \ge R_0(\epsilon).
		\ee
		Hence the inequality \eqref{casebs1n} implies that 
		\begin{equation}
		    \begin{aligned}
		        Y(R) &\le C_{14}R^{\frac{1}{p}}\left(\|\Bu\|_{L^{\infty}(\OR)}+\|\Bu\|_{L^{\infty}(\OR)}^{2}\right)[Y'(R)]^{\frac{p-1}{p}}\\
                &\le  C_{15} e^{\frac{R}{3C_8 p}}\cdot \epsilon^2 e^{\frac{2R}{3C_8 p}}[Y'(R)]^{\frac{p-1}{p}} \le  \epsilon C_{15}e^{\frac{R}{C_8 p}}[Y'(R)]^{\frac{p-1}{p}} ,
		    \end{aligned}
		\end{equation}
		and one obtains
		\begin{equation}\label{finalb2}
			\begin{aligned}
				\frac{Y'(R)}{[Y(R)]^{\frac{p}{p-1}}} \ge [\epsilon C_{15    }]^{-\frac{p}{p-1}}e^{-\frac{R}{C_8 p}\cdot \frac{p}{p-1}}.
			\end{aligned}  
		\end{equation}
		If $\Bu$ is not identically equal zero, according to $Case (b_2)$ of Propositon \ref{IG1}, $Y(R)$ must be unbounded as $R \rightarrow +\infty$. For every $R$ sufficiently large, integrating \eqref{finalb2} over $[R,+\infty)$ one arrives at 
		\be\label{ppeess1} 
		Y(R)\cdot e^{-\frac{R}{C_8}} \le C(\epsilon C_{15})^p.
		\ee
		Since $\epsilon$ can be arbitrarily small, this implies \eqref{pes2} and leads to a contradiction with the assumption that $\Bu$ is not identically zero. This completes the proof of $Case(a_2)$ of Theorem \ref{th:01}.

	\end{proof}

	\section{GENERAL NON-NEWTONIAN FLUIDS IN A SLAB }\label{Sec4}
	
	This section is devoted to the study for the general solution of the non-Newtonian fluids \eqref{eqsteadynNfs}
	in a slab with no-slip boundary conditions \eqref{noslipboun}.  Before giving the proof of Theorem \ref{th:02}, we present a preliminary proposition.
\begin{pro}
    \label{IG2}
		Let $\Bu$ be a weak solution to the equations \eqref{eqsteadynNfs} in $\Omega=\mathbb{R}^2 \times (0,1)$ with no-slip boundary conditions \eqref{noslipboun}. Let $\mathcal{E}(R)$ be defined as in \eqref{D-integral}. Then the following holds:
        
		
		 $(d_1)$ For $2 \le p<3$, if $\Bu$ satisfies 
		\begin{equation}\label{3Da1}
			\varliminf_{R\rightarrow +\infty} (\ln R)^{-\frac{3}{3-p}}\mathcal{E}(R)=0,
		\end{equation}
        then $\Bu \equiv 0$.
        
		$(d_2)$ For  $p \ge 3$, there exists a constant $C_{\ast\ast}(p)>0$, such that if $\Bu$ satisfies   
		\begin{equation}\label{3Da2}
			\varliminf_{R\rightarrow +\infty}R^{-\frac{1}{C_{\ast\ast}(p)}}\mathcal{E}(R)=0,
		\end{equation}
        then $\Bu \equiv 0$.
\end{pro}
\begin{proof}
Instead of the cut-off function $\varphi_{R}(r)$ in Section \ref{Sec3}, one introduces a new cut-off function $\tau (s)$ satisfying 
	\begin{equation}\label{nco}
		\tau(s) = \left\{ \ba
		&1,\ \ \ \ \ \ \ \ \ \ 0\le s < \frac{1}{2}, \\
		&-2s+2,\ \ \ \ \frac{1}{2} \leq s \leq 1, \\
		&0, \ \ \ \ \ \ \ \ \ \ s > 1.
		\ea  \right.
	\end{equation}
	Set $\tau_R(x')=\tau(\frac{|x'|}{R})$, where $x'=(x_1,x_2)\in \mathbb{R}^2$ and $R$ is a large positive number. Multiplying the momentum equation  in \eqref{eqsteadynNfs}
	with $\Bu\tau_{R}$, integrating by parts yields
	\begin{equation}\label{EQF1}
		\begin{split}
			&\int_{\Omega}(1+|\D(\Bu)|^2)^{\frac{p-2}{2}}|\D(\Bu)|^2\tau_{R}\,\dd\Bx\\
			=&
			-\int_{\Omega}(1+|\D(\Bu)|^2)^{\frac{p-2}{2}}\D(\Bu):(\Bu\otimes\nabla\tau_{R})\,\dd\Bx+\int_{\Omega}\frac{1}{2}|\Bu|^{2}\Bu\cdot
			\nabla\tau_{R}\,\dd\Bx
			+\int_{\Omega}P\Bu\cdot\nabla\tau_{R}\,\dd\Bx.
		\end{split}
	\end{equation}
Analogous to \eqref{Es}, one can derive from \eqref{EQF1} that
\begin{equation}\label{Es1}
			\begin{split}
				C\int_{\Omega}(|\nabla\Bu|^{2}+|\D(\Bu)|^p)\tau_{R}\,\dd\Bx &\le  \int_{\Omega}(1+|\D(\Bu)|^2)^{\frac{p-2}{2}}|\D(\Bu)|^2\tau_{R}\,\dd\Bx\\
               &\quad +\int_{\Omega}\D(\Bu):(\Bu \otimes \nabla \tau_{R})\,\dd\Bx-\int_{\Omega} \nabla\tau_{R} \cdot\nabla\Bu \cdot  \Bu\,\dd\Bx \\
				&\leq \left|\int_{\Omega}(1+|\D(\Bu)|^2)^{\frac{p-2}{2}}\D(\Bu):(\Bu\otimes\nabla\tau_{R})\,\dd\Bx\right|\\
                 &\quad +\left|  \int_{\Omega}\D(\Bu):(\Bu \otimes \nabla \tau_{R})\,\dd\Bx\right|+\left|\int_{\Omega} \nabla\tau_{R} \cdot \nabla\Bu \cdot \Bu\,\dd\Bx  \right|\\
                &\quad +\left|\int_{\Omega}\frac{1}{2}|\Bu|^{2}\Bu\cdot
				\nabla\tau_{R}\,\dd\Bx\right|
                +\left|\int_{\Omega}P\Bu\cdot\nabla\tau_{R}\,\dd\Bx\right|
               .
			\end{split}
		\end{equation}

$Case\,(c_1)$. $2\le p <3$. One employs the same technique as that for the proof of Theorem \ref{th:01}. First, due to the no-slip boundary conditions \eqref{noslipboun}, one can make zero extension $\Bu$ from $Z_R$ to $Q_R$, and by Korn's inequality \eqref{ki2}, one has 
		\begin{equation}\label{utoDu}
		    \begin{aligned}
		        \|\Bu\|_{L^p(Z_R)}=\|\Bu\|_{L^p(Q_R)} \le \|\nabla \Bu\|_{L^p(Q_R)}\le C\|\D(\Bu)\|_{L^p(Q_R)}= C \|\D(\Bu)\|_{L^p(Z_R)},
		    \end{aligned}
		\end{equation}
		where the constant $C$ is independent of $R$.
        By H{\"o}lder's inequality, Poincar\'e  inequality and  \eqref{utoDu}, one  obtains
		\begin{equation}\label{First term 3D}
			\begin{aligned}
				&\left|\int_{\Omega}(1+|\D(\Bu)|^2)^{\frac{p-2}{2}}\D(\Bu):  (\Bu\otimes\nabla\tau_{R})\,\dd\Bx\right|\\ 	\leq &\,
				\int_{Z_R}(|\D(\Bu)|+|\D(\Bu)|^{p-1}):\left|(\Bu\otimes\nabla\tau_{R})\right|\,\dd\Bx\\
			\le&\,\frac{C}{R} \left(\|\nabla\Bu \|_{L^2(Z_R)}\cdot \|\Bu\|_{L^2(Z_R)}+\|\D(\Bu)\|_{L^p(Z_R)}^{p-1}\cdot \|\Bu\|_{L^p(Z_R)}\right) \\
				\le& \,\frac{C}{R}(\| \nabla \Bu\|_{L^2(Z_R)}^2+ \|\D(\Bu)\|_{L^p(Z_R)}^p)
			\end{aligned}
		\end{equation}
and 
\begin{equation}\label{fj2}
    \begin{aligned}
       & \left|  \int_{\Omega}\D(\Bu):(\Bu \otimes \nabla \tau_R)\,\dd\Bx\right|+\left|\int_{\Omega}\nabla\tau_R \cdot \nabla\Bu \cdot  \Bu\,\dd\Bx  \right|\\ \le&\, \frac{C}{R}\left( \|\D(\Bu)\|_{L^2(Z_R)}+\|\nabla \Bu\|_{L^2(Z_R)}\right)\cdot \|\Bu\|_{L^2(Z_R)}
        \le\,\frac{C}{R}\| \nabla \Bu\|_{L^2(Z_R)}^2.
    \end{aligned}
\end{equation}
    Let $q=\frac{2p}{p-1}$,  by the 
    Gagliardo-Nirenberg interpolation inequality \eqref{GN} and \eqref{utoDu}, one has
    \begin{equation}\label{L2Pestims}
    \begin{aligned}
  \|\Bu\|_{L^{\frac{2p}{p-1}}(Z_R)}=\|\Bu\|_{L^q(Q_R)} \le 
  C\|\nabla\Bu\|^{\theta}_{L^p(Q_R)}  \|\Bu\|^{1-\theta}_{L^p(Q_R)}+C\|\Bu\|_{L^p(Q_R)}\le
  C\|\D(\Bu)\|_{L^p(Z_R)},
    \end{aligned}
    \end{equation}
    where $\theta=\frac{3(q-p)}{pq}$.
		By Poincaré inequality,   \eqref{utoDu} and \eqref{L2Pestims} yields
		\begin{equation}\label{non-linear term 3D}
			\begin{aligned}
				\left| \int_{\Omega} \frac{1}{2} |\Bu|^2 \Bu \cdot \nabla \tau_R \,\dd\Bx\right|
                &\le \frac{C}{R}\|u^r\|_{L^p(Z_R)}\cdot \|\Bu\|^2_{L^{\frac{2p}{p-1}}(Z_R)}\\
				&\le \frac{C}{R}\|\nabla \Bu\|_{L^p(Z_R)}\cdot \|\D(\Bu)\|^2_{L^{p}(Z_R)}\\
                &\le  \frac{C}{R}\|\D(\Bu)\|_{L^p(Z_R)}^3,
			\end{aligned}
		\end{equation}
		where $C$ is independent of $R$.
		
		For the estimate of the pressure term in \eqref{Es1},  one obtains
        \[
     \int_{\Omega}P\Bu\cdot\nabla\tau_{R}\,\dd\Bx=-\frac{1}{R}\int_{Z_{R}}Pu^{r}\,\dd\Bx.
        \]
   The divergence free condition for general  solutions in cylindrical coordinates is
   \[
  \p_{r}u^{r}+\frac{u^{r}}{r}+\frac{\p_{\theta}u^{\theta}}{r}+\p_{z}u^{z}=0.
   \]
Hence, for fixed $r\geq0$, one obtains
\begin{equation}
\p_{r}\int_{0}^{1}\int_{0}^{2\pi}ru^{r}\,\dd\theta\dd z=-\int_{0}^{1}\int_{0}^{2\pi}\p_{\theta}u^{\theta}+\p_{z}(ru^{z})\,\dd\theta\dd z=0.
\end{equation}
 Thus, it holds that
 \begin{equation}
\int_{0}^{1}\int_{0}^{2\pi}ru^{r}\,\dd\theta\dd z=0 \,\,\,\, \text{and}  \,\,\,\,
\int_{0}^{1}\int_{0}^{2\pi}\int_{\frac{R}{2}}^{R}u^{r}\,\dd r\dd\theta\dd z=0.
 \end{equation}
By Bogovskii map of Lemma \ref{Bogovskii}, there is a vector valued function $\bBV\in W^{1,p}_{0}(Z_{R},\mathbb{R}^3)$, such that
      \begin{equation}
        {\rm div}\bBV=u^{r}, \quad \text{in} \, Z_{R},
          \end{equation}
          with the estimate
		\begin{equation}\label{esbm3D}
			\|\nabla \bBV\|_{L^p(Z_R)}\leq CR\|u^{r}\|_{L^{p}(Z_R)},
		\end{equation}
		where $C$ is independent of $R$.

		Therefore, combining \eqref{esbm3D}, Poincar\'e inequality, H{\"o}lder's inequality  and \eqref{L2Pestims}, one obtains 
		\begin{equation}\label{3Da1esp}
			\begin{split}
				\left|\int_{\Omega}P\Bu\cdot\nabla\tau_{R}\,\dd\Bx\right| &\le \frac{C}{R}\left|  \int_{Z_R} P u^r\,\dd\Bx \right|= \frac{C}{R}\left|  \int_{Z_R} P \text{div} \bBV \,\dd\Bx \right|= \frac{C}{R}\left|  \int_{Z_R} -\nabla P \cdot \bBV\,\dd\Bx \right| \\
				&= \frac{C}{R}\left|  \int_{Z_R}     [-\text{div} A_p(\Bu)+(\Bu \cdot \nabla )\Bu ]   \cdot \bBV\,\dd\Bx \right|\\
				&= 	 \frac{C}{R}\left|\int_{Z_R}(1+|\D(\Bu)|^2)^{\frac{p-2}{2}}\D(\Bu):\D \bBV \,\dd\Bx-\int_{Z_R}(\Bu \cdot \nabla )\bBV \cdot \Bu\,\dd\Bx \right| \\
				&\le \frac{C}{R} \left(R\|\D(\Bu)\|_{L^2(Z_R)}\cdot \|u^r\|_{L^2(Z_R)}+R\|\D(\Bu)\|_{L^p(Z_R)}^{p-1}\cdot \|u^r\|_{L^p(Z_R)} \right)\\  
				&\quad + \frac{C}{R}\|\nabla \bBV\|_{L^p(Z_R)}\cdot \|\Bu\|^2_{L^{\frac{2p}{p-1}}(Z_R)}\\
                &\le  C\left(\|\nabla \Bu\|_{L^2(Z_R)}^2+\|\D(\Bu)\|_{L^p(Z_R)}^p  \right)
				+ C\|u^r\|_{L^p(Z_R)}\cdot \|\Bu\|^2_{L^{\frac{2p}{p-1}}(Z_R)}\\
                &\le  C\left(\|\nabla \Bu\|_{L^2(Z_R)}^2+\|\D(\Bu)\|_{L^p(Z_R)}^p  \right)+C\|\D(\Bu)\|_{L^p(Z_R)}^3.
			\end{split}
		\end{equation} 
     Combining \eqref{First term 3D}-\eqref{fj2}, \eqref{non-linear term 3D} and  \eqref{3Da1esp}, one arrives at
		\begin{equation}\label{3Da1f}
			\int_{\Omega}(|\nabla\Bu|^{2}+|\D(\Bu)|^p)\tau_{R}\,\dd\Bx \le  C \left(\|\nabla \Bu\|_{L^2(Z_R)}^2+\|\D(\Bu)\|_{L^p(Z_R)}^p  \right)+C\|\D(\Bu)\|_{L^p(Z_R)}^3.
		\end{equation}
		
		Let 
		\begin{equation}\label{ZR}
			\begin{aligned}
				\mathcal{Y}(R)=
				\int_{0}^{1}\iint_{\mathbb{R}^{2}}
				(|\nabla\Bu|^{2}+|\D(\Bu)|^p)\tau_{R}\left(\sqrt{x_{1}^{2}+x_{2}^{2}}\right)
				\, \dd x_{1} \dd x_{2} \dd x_{3} ,
			\end{aligned}
		\end{equation}
		and  straightforward computations give
		\[
		\mathcal{Y}^{\prime}(R) =\frac{C}{R}\int_{Z_R}(|\nabla\Bu|^{2}+|\D(\Bu)|^p )\,\dd\Bx.
		\]
		Hence the estimate \eqref{3Da1f} can be written as 
		\begin{equation}\label{3DCasea1es}
				\mathcal{Y}(R)\le \widetilde{C}_1R	\mathcal{Y}^{\prime}(R)+\widetilde{C}_2R^{\frac{3}{p}}	[\mathcal{Y}^{\prime}(R)]^{\frac{3}{p}}.  
		\end{equation}
    Then $Case\,(c_1)$ of Proposition \ref{IG2} follows from Lemma \ref{SV}.\\
		$Case\,(c_2)$. $p\ge 3$. If $p>3$, the estimate \eqref{First term 3D}-\eqref{fj2} is similar to $Case\,(c_1)$, except for a different approach in handling \eqref{non-linear term 3D} and \eqref{3Da1esp}, one has 
		\begin{equation}
			\begin{split}\label{3Da2 non linear term}
				\left|\int_{\Omega}\frac{1}{2}|\Bu|^{2}\Bu\cdot
				\nabla\tau_{R}\,\dd\Bx\right|&\le \frac{C}{R}\|\Bu\|^3_{L^3(Q_R)}
				\le \frac{C}{R}\|\Bu\|_{L^2(Q_R)}^{3\theta}\cdot \|\Bu\|_{L^p(Q_R)}^{3(1-\theta)} \\
				&\le  \frac{C}{R} \left( \|\Bu\|^2_{L^2(Z_R)}+\|\Bu\|^{\frac{6(1-\theta)}{2-3\theta}}_{L^p(Z_R)}\right)
				\le\frac{C}{R} \left( \|\nabla\Bu\|^2_{L^2(Z_R)}+\|\D(\Bu)\|^{p}_{L^p(Z_R)}\right),
			\end{split}
		\end{equation}
		where $\theta= \frac{2(p-3)}{3(p-2)}$.\\
		The similar computations as \eqref{3Da1esp} show that 
		\begin{equation}\label{3Da2esp}
			\begin{split}
				\left|\int_{\Omega}P\Bu\cdot\nabla\tau_{R}\,\dd\Bx\right| &\le \frac{C}{R} \left(R\|\D(\Bu)\|_{L^2(Z_R)}\cdot \|u^r\|_{L^2(Z_R)}+R\|\D(\Bu)\|_{L^p(Z_R)}^{p-1}\cdot \|u^r\|_{L^p(Z_R)} \right)\\   
				&\quad + C\|u^r\|_{L^p(Z_R)}\cdot \|\Bu\|^2_{L^{\frac{2p}{p-1}}(Z_R)}\\
				&\le  C\left(\|\nabla \Bu\|_{L^2(Z_R)}^2+\|\D(\Bu)\|_{L^p(Z_R)}^p  \right)+C\|\Bu\|_{L^2(Q_R)}^{2\theta_1}\cdot \|\Bu\|_{L^p(Q_R)}^{2(1-\theta_1)+1}\\
				&\le  C\left(\|\nabla \Bu\|_{L^2(Z_R)}^2+\|\D(\Bu)\|_{L^p(Z_R)}^p  \right)+C\left(\|\Bu\|_{L^2(Q_R)}^{2}+\|\Bu\|_{L^p(Q_R)}^{\frac{3-2\theta_1}{1-\theta_1}}\right)\\
				&\le C\left(\|\nabla \Bu\|_{L^2(Z_R)}^2+\|\D(\Bu)\|_{L^p(Z_R)}^p  \right),
			\end{split}
		\end{equation}
		where $\theta_1=\frac{p-3}{p-2}$.
		Combining \eqref{First term 3D}-\eqref{fj2}, \eqref{3Da2 non linear term} and \eqref{3Da2esp}, one derives 
		\begin{equation}\label{3DCasea2es}
			\mathcal{Y}(R)\le \widetilde{C}_3 R\mathcal{Y}^{\prime}(R).   
		\end{equation}

    For $p=3$. By the estimate \eqref{First term 3D}-\eqref{non-linear term 3D} and \eqref{3Da1esp}, one obtains the same estimate as \eqref{3DCasea2es}. Then $Case \,(c_2)$ of Proposition \ref{IG2} also follows from Lemma \ref{SV}.
\end{proof}
\begin{remark}
    The proof of Theorem \ref{DS} is essentially contained in that of Proposition \ref{IG2}, specifically, in the derivation of \eqref{3Da1f} and \eqref{3DCasea2es}.
\end{remark}

Now we are ready to  prove Theorem \ref{th:02}.
	\begin{proof}[Proof for Theorem \ref{th:02}]
		$Case \,(1)$. $ 2\le p<3$. One needs to estimate \eqref{First term 3D}-\eqref{fj2}, \eqref{non-linear term 3D} and \eqref{3Da1esp} in a different way. Using H{\"o}lder's inequality yields
		\begin{equation}\label{First term 3D d1}
			\begin{aligned}
				& \left|\int_{\Omega}(1+|\D(\Bu)|^2)^{\frac{p-2}{2}}\D(\Bu):(\Bu\otimes\nabla\tau_{R})\,\dd\Bx\right|
                \\
                \leq\, &
				C\int_{Z_R}(|\D(\Bu)|+|\D(\Bu)|^{p-1}):|\Bu\otimes\nabla\tau_{R}|\,\dd\Bx\\
					\le\, & \frac{C}{R} \left(\|\D(\Bu)\|_{L^2(Z_R)}\cdot \|\Bu\|_{L^2(Z_R)}+\|\D(\Bu)\|_{L^p(Z_R)}^{p-1}\cdot \|\Bu\|_{L^p(Z_R)} \right)\\
				\le \,& C\|\Bu\|_{L^{\infty}(Z_R)}\cdot\left( \| \nabla \Bu\|_{L^2(Z_R)}+ R^{\frac{2}{p}-1}\|\D(\Bu)\|_{L^p(Z_R)}^{p-1}\right)    	\end{aligned}
		\end{equation}
and 
\begin{equation}\label{fj3}
    \begin{aligned}
       & \left|  \int_{\Omega}\D(\Bu):\Bu \otimes \nabla \tau_R\,\dd\Bx\right|+\left|\int_{\Omega} \nabla\tau_R \cdot\nabla\Bu \cdot  \Bu\,\dd\Bx  \right|\\
       \le\, & \frac{C}{R} \left(\|\D(\Bu)\|_{L^2(Z_R)}+\|\nabla\Bu\|_{L^2(Z_R)}\right)\cdot \|\Bu\|_{L^2(Z_R)}\\
        \le\, & C\|\Bu\|_{L^{\infty}(Z_R)}\cdot\| \nabla \Bu\|_{L^2(Z_R)}.
    \end{aligned}
\end{equation}
		Instead of \eqref{non-linear term 3D}, one has 
		\begin{equation}\label{non-linear term 3D d1}
			\begin{aligned}
				\left| \int_{\Omega} \frac{1}{2} |\Bu|^2 \Bu \cdot \nabla \tau_R \,\dd\Bx\right| &\leq \frac{C}{R} \|u^r\|_{L^{\infty}(Z_R)}\|\Bu\|^{3-p}_{L^{\infty}(Z_R)}\int_{Z_R}|\Bu|^{p-1}\, \dd \Bx\\
				&\le CR^{\frac{2}{p}-1} \|u^r\|_{L^{\infty}(Z_R)}\|\Bu\|^{3-p}_{L^{\infty}(Z_R)}\|\D(\Bu)\|^{p-1}_{L^p(Z_R)}.
			\end{aligned}
		\end{equation}
		For the estimate of the pressure term, one obtains
		\begin{equation}\label{3Dd1esp}
			\begin{split}	\left|\int_{\Omega}P\Bu\cdot\nabla\tau_{R}\,\dd\Bx\right| 	&\le \frac{C}{R} \left(R\|\nabla\Bu\|_{L^2(Z_R)}\cdot \|u^r\|_{L^2(Z_R)}+R\|\D(\Bu)\|_{L^p(Z_R)}^{p-1}\cdot \|u^r\|_{L^p(Z_R)} \right)\\ 
				&\quad + C\|u^r\|_{L^p(Z_R)}\cdot \|\Bu\|^2_{L^{\frac{2p}{p-1}}(Z_R)}\\
                &\le C\left(R\|u^r\|_{L^\infty(Z_R)}\cdot\|\nabla\Bu\|_{L^2(Z_R)}
                +R^{\frac{2}{p}}\|u^r\|_{L^\infty(Z_R)} \cdot\|\D(\Bu)\|_{L^p(Z_R)}^{p-1} \right)\\ 
				&\quad + CR^{\frac{2}{p}}\|u^r\|_{L^\infty(Z_R)}\cdot \|\Bu\|^{3-p}_{L^\infty(Z_R)}\left(\int_{Z_{R}}|\Bu|^{p}\,\dd\Bx\right)^{\frac{p-1}{p}}\\
				&\le  C\left(R\|u^r\|_{L^{\infty}(Z_R)}\cdot\|\nabla \Bu\|_{L^2(Z_R)}+R^{\frac{2}{p}}\|u^r\|_{L^{\infty}(Z_R)}\cdot\|\D(\Bu)\|_{L^p(Z_R)}^{p-1}  \right)
				\\
				&\quad + CR^{\frac{2}{p}}\|u^r\|_{L^{\infty}(Z_R)}\|\Bu\|^{3-p}_{L^{\infty}(Z_R)}\cdot\| \D(\Bu)\|_{L^p(Z_R)}^{p-1}.\\
			\end{split}
		\end{equation}
		Combining \eqref{First term 3D d1}-\eqref{3Dd1esp}, one derives
        \begin{equation}
        \begin{split}
           & \int_{\Omega}(|\nabla\Bu|^{2}+|\D(\Bu)|^p)\tau_{R}\,\dd\Bx\\
           \le &
            C\left(\|\Bu\|_{L^{\infty}(Z_R)}+R\|u^r\|_{L^{\infty}(Z_R)}\right)\|\nabla\Bu\|_{L^{2}(Z_R)}\\
             &+C\left(R^{\frac{2}{p}-1}\|\Bu\|_{L^{\infty}(Z_R)}+R^{\frac{2}{p}}\|u^r\|_{L^{\infty}(Z_R)}\|\Bu\|^{3-p}_{L^{\infty}(Z_R)}\right)\| \D(\Bu)\|_{L^p(Z_R)}^{p-1}.
             \end{split}
        \end{equation}
Hence, one arrives at
		\begin{equation}\label{d1}
			\begin{aligned}
				\mathcal{Y}(R) &\le C\left(  R^{\frac{1}{2}}\|\Bu\|_{L^{\infty}(Z_R)}+R^{\frac{3}{2}}\|u^r\|_{L^{\infty}(Z_R)}\right)[\mathcal{Y}^{\prime}(R)]^{\frac{1}{2}}\\
				&\quad +C \left(R^{\frac{1}{p}}\|\Bu\|_{L^{\infty}(Z_R)}+R^{\frac{p+1}{p}}\|u^r\|_{L^{\infty}(Z_R)}\|\Bu\|^{3-p}_{L^{\infty}(Z_R)}\right)[\mathcal{Y}^{\prime}(R)]^{\frac{p-1}{p}},
			\end{aligned}
		\end{equation}
  where $\mathcal{Y}(R)$ is defined in \eqref{ZR}.   
  
	 Suppose $\mathcal{Y}(R)$ is not identically zero,  it follows from  Proposition \ref{IG2} that there exist $C_0$ and $R_0$ such that
		\begin{equation}\label{eqZRes}
		    \begin{aligned}
		       \mathcal{Y}(R)\ge C_0 (\ln R)^{\frac{3}{3-p}},  \quad \text{for any} \, R \ge R_0.
		    \end{aligned}
		\end{equation}
		Then from \eqref{3DCasea1es} one has 
		\[
		C_0 (\ln R)^{\frac{3}{3-p}} \le \widetilde{C}_1 R\mathcal{Y}^{\prime}(R)+\widetilde{C}_2 
        R^{\frac{3}{p}}[\mathcal{Y}^{\prime}(R)]^{\frac{3}{p}}.
		\]
	It holds that
		\begin{equation}
			\mathcal{Y}^{\prime}(R) \ge C R^{-1}(\ln R)^{\frac{p}{3-p}}.
		\end{equation}
		Utilizing \eqref{3Db}, $\mathcal{Y}(R)$ in estimate \eqref{d1}  can be reduced to
		\begin{equation}\label{3Dd1}
			\begin{aligned}
				\mathcal{Y}(R) &\le C R^{\frac{1}{2}}[\mathcal{Y}^{\prime}(R)]^{\frac{1}{p}-\frac{1}{2}}\cdot[\mathcal{Y}^{\prime}(R)]^{\frac{p-1}{p}}+C R^{\frac{1}{p}}[\mathcal{Y}^{\prime}(R)]^{\frac{p-1}{p}}\\
                & \le CR^{\frac{p-1}{p}}(\ln R)^{\frac{2-p}{2(3-p)}}[\mathcal{Y}'(R)]^{\frac{p-1}{p}}+C R^{\frac{1}{p}}[\mathcal{Y}^{\prime}(R)]^{\frac{p-1}{p}}\\
                &\le CR^{\frac{p-1}{p}}[\mathcal{Y}^{\prime}(R)]^{\frac{p-1}{p}}.
			\end{aligned}
		\end{equation}
     Assume that $\Bu$ is not identically equal to zero. Then, by $Case \,(c_1) $ Proposition \ref{IG2}, $\mathcal{Y}(R)$ satisfies the lower bound established in \eqref{eqZRes}. So  $\mathcal{Y}(R)$ must be unbounded as $R \rightarrow +\infty$. For every $R_1\ge R_0$, integrating \eqref{3Dd1} over $[R_1,R_2)$, one arrives at
	\[
		C\ln \frac{R_2}{R_1} \le [\mathcal{Y}(R_1)]^{-\frac{1}{p-1}}-[\mathcal{Y}(R_2)]^{-\frac{1}{p-1}}\le
        [\mathcal{Y}(R_1)]^{-\frac{1}{p-1}}.
		\]
This leads to a contradiction  when $R_2$ is sufficiently large. Therefore,  $\Bu \equiv 0$.

		$Case\,(2)$. $p \ge 3$. Estimates for the first term on the right hand side of \eqref{Es1} are the same to \eqref{First term 3D d1}-\eqref{fj3}, one needs to estimate \eqref{non-linear term 3D d1}, \eqref{3Dd1esp} in a different way. Using Poincaré inequality yields
		\begin{equation}\label{non-linear term 3D d2}
			\begin{aligned}
				\left| \int_{\Omega} \frac{1}{2} |\Bu|^2 \Bu \cdot \nabla \tau_R \,\dd\Bx\right| &\leq \frac{C}{R} \|u^r\|_{L^{\infty}(Z_R)}\|\Bu\|_{L^{\infty}(Z_R)}\int_{Z_R}|\Bu|\, \dd \Bx\\
				&\le C \|u^r\|_{L^{\infty}(Z_R)}\|\Bu\|_{L^{\infty}(Z_R)}\|\nabla\Bu\|_{L^2(Z_R)}.
			\end{aligned}
		\end{equation}
		The similar computations as \eqref{3Dd1esp} show that
		\begin{equation}\label{3Dd2esp}
			\begin{aligned}
				\left|\int_{\Omega}P\Bu\cdot\nabla\tau_{R}\,\dd\Bx\right| 	&\le \frac{C}{R} \left(R\|\nabla\Bu\|_{L^2(Z_R)}\cdot \|u^r\|_{L^2(Z_R)}+R\|\D(\Bu)\|_{L^p(Z_R)}^{p-1}\cdot \|u^r\|_{L^p(Z_R)} \right)\\  
				&\quad + C\|u^r\|_{L^2(Z_R)}\cdot \|\Bu\|^2_{L^{4}(Z_R)}\\
				&\le  C\left(R\|u^r\|_{L^{\infty}(Z_R)}\cdot\|\nabla \Bu\|_{L^2(Z_R)}+R^{\frac{2}{p}}\|u^r\|_{L^{\infty}(Z_R)}\cdot\|\D(\Bu)\|_{L^p(Z_R)}^{p-1}  \right)
				\\
				&\quad + CR\|u^r\|_{L^{\infty}(Z_R)}\|\Bu\|_{L^{\infty}(Z_R)}\cdot\| \nabla\Bu\|_{L^2(Z_R)}.\\
			\end{aligned}
		\end{equation}
		Combining \eqref{First term 3D d1}-\eqref{fj3}, \eqref{non-linear term 3D d2} and \eqref{3Dd2esp}, one derives 
        \begin{equation}
            \begin{aligned}
                & \int_{\Omega}(|\nabla\Bu|^{2}+|\D(\Bu)|^p)\tau_{R}\,\dd\Bx\\ \le&  C\left(R\|u^r\|_{L^{\infty}(Z_R)}+\|\Bu\|_{L^{\infty}(Z_R)}+R\|u^r\|_{L^{\infty}(Z_R)}\|\Bu\|_{L^{\infty}(Z_R)}\right)\cdot\|\nabla \Bu\|_{L^2(Z_R)}
				\\
				&+C\left(R^{\frac{2}{p}}\|u^r\|_{L^{\infty}(Z_R)}+R^{\frac{2}{p}-1} \|\Bu\|_{L^{\infty}(Z_R)}\right)\cdot\|\D(\Bu)\|_{L^p(Z_R)}^{p-1}.
            \end{aligned}
        \end{equation}
Hence, one obtains 
		\begin{equation}\label{d2}
			\begin{aligned}
				\mathcal{Y}(R) &\le C \left(R^{\frac{3}{2}}\|u^r\|_{L^{\infty}(Z_R)}+ R^{\frac{1}{2}}\|\Bu\|_{L^{\infty}(Z_R)}+R^{\frac{3}{2}}\|u^r\|_{L^{\infty}(Z_R)}\|\Bu\|_{L^{\infty}(Z_R)}\right)[\mathcal{Y}^{\prime}(R)]^{\frac{1}{2}}\\
				&\quad +C \left( R^{\frac{p+1}{p}}\|u^r\|_{L^{\infty}(Z_R)}+R^{\frac{1}{p}}\|\Bu\|_{L^{\infty}(Z_R)}\right)[\mathcal{Y}^{\prime}(R)]^{\frac{p-1}{p}}.
			\end{aligned}
		\end{equation}
    
 Suppose $\mathcal{Y}(R)$ is not identically zero, it follows from Proposition \ref{IG2} that there exist $R_0$ and $C_{\ast \ast(p)}=\widetilde{C}_3$ such that
		\[
		\mathcal{Y}(R)\ge C_0 R^{\frac{1}{\widetilde{C}_3}},  \quad \text{for any } \, R \ge R_0.
		\]
    	Combining  with \eqref{3DCasea2es} yields
		\[
		C_0 R^{\frac{1}{\widetilde{C}_3}} \le \widetilde{C}_3 R\mathcal{Y}^{\prime}(R),
		\]
		i.e.,
		\begin{equation}
			\mathcal{Y}^{\prime}(R) \ge \frac{C_0}{\widetilde{C}_3}R^{\frac{1}{\widetilde{C}_3}-1}.
		\end{equation}
		Utilizing \eqref{3Db}, $\mathcal{Y}(R)$ in estimate \eqref{d2}  can be reduced to
		\begin{equation}\label{3Dd11}
			\begin{aligned}
				\mathcal{Y}(R) &\le C R^{\frac{1}{2}}[\mathcal{Y}^{\prime}(R)]^{\frac{1}{p}-\frac{1}{2}}\cdot[\mathcal{Y}^{\prime}(R)]^{\frac{p-1}{p}}+C R^{\frac{1}{p}}[\mathcal{Y}^{\prime}(R)]^{\frac{p-1}{p}}\\
                & \le CR^{\frac{p-1}{p}+\frac{1}{\widetilde{C}_3}(\frac{1}{p}-\frac{1}{2})}[\mathcal{Y}^{\prime}(R)]^{\frac{p-1}{p}}+C R^{\frac{1}{p}}[\mathcal{Y}^{\prime}(R)]^{\frac{p-1}{p}}\\
                &\le CR^{\frac{p-1}{p}}[\mathcal{Y}^{\prime}(R)]^{\frac{p-1}{p}}.
			\end{aligned}
		\end{equation}
A similar argument to the proof of $Case\,(1)$ applies here for $p \ge 3$, showing that $\Bu$ is a zero vector. This completes the proof of Theorem \ref{th:02}.

	\end{proof}

	{\bf Acknowledgement.}
The authors thank Professor Yun Wang for valuable discussions and suggestions.


\medskip
	
	\begin{thebibliography}{99}
		
		
		
		\bibitem{aBGWX} J. Bang,  C. Gui,  Y. Wang,  C. Xie,
		Liouville-type theorems for steady solutions to the Navier-Stokes system in a slab,
		{\it J. Fluid Mech.},  {\bf 1005}(2025),  A6.
		
	 \bibitem{aBYA} J. Bang,  Z. Yang,
	Saint-Venant estimates and Liouville-type theorems for the stationary Navier-Stokes equation in $\mathbb{R}^{3}$,
	{\it J. Math. Fluid Mech.}, {\bf27}(2025), no. 3, Paper No. 35, 13 pp.
		
		
		
		
		
		\bibitem{JMFM13bfz} M. Bildhauer, M. Fuchs, G. Zhang,
		\newblock Liouville-type theorems for steady flows of degenerate power law fluids in the plane,
		\newblock \emph{J. Math. Fluid Mech.}, \textbf{15}(2013), no. 3, 583--616.
		
		
		\bibitem{BM} M.~E. Bogovski\u i, Solution of the first boundary value problem for an equation of continuity of an incompressible medium, {\it Dokl. Akad. Nauk SSSR.}, {\bf 248}(1979), no.~5, 1037--1040.
		
		
		\bibitem{BPZZ}B. Carrillo, X. H. Pan, Q. S. Zhang,  N. Zhao,  Decay and vanishing of some D-solutions of the
		Navier-Stokes equations, 
		{\it Arch. Ration. Mech. Anal.}, {\bf 237}(2020), no. 3, 1383--1419.
		
		
		
		
		
		
		
		
		
		\bibitem{CCMP14}D. Chae, Liouville-type theorems for the forced Euler equations and the Navier-Stokes equations, {\it Comm. Math. Phys.}, {\bf 326}(2014), no. 1, 37--48.
		
		
		
		\bibitem{CWDCDS16} D. Chae, S. Weng,  Liouville type theorems for the steady axially symmetric Navier-Stokes and Magnetohydrodynamic equations,
		{\it Discrete Contin. Dyn. Syst.}, {\bf36}(2016), 5267--5285.
		
		
		
		
		\bibitem{CW2019CVPDE} D. Chae, J. Wolf, On Liouville type theorem for the stationary Navier-Stokes equations,
		{\it Calc. Var. Partial Differential Equations}, {\bf58}(2019), Paper No. 111, 11.
		
		
		
		\bibitem{DJ1} D. Chae, J. Wolf, On Liouville type theorem for stationary non-Newtonian fluid equations, {\it J. Nonlinear Sci.},  \textbf{30}(2020), no. 4, 1503--1517.
		
		
		\bibitem{DJJDE21} D. Chae, J. Kim, J. Wolf, On Liouville type theorems in the stationary non-Newtonian fluids, {\it J. Differential Equations}, \textbf{302}(2021), 710--727.
		
		
		
		
		
		
		
		
		
		
		
		
		
		
		
		\bibitem{FJMJSM} J. Frehse, J. M\'{a}lek, M. Steinhauer, An existence result for fluids with shear dependent viscosity-steady flows, {\it Nonlinear Anal.}, \textbf{30}(1997), 3041--3049. 
		
		\bibitem{MF}  M. Fuchs, Liouville theorems for stationary flows of shear thickening fluids in the plane, {\it J. Math. Fluid Mech.} \textbf{14}(2012), no. 3, 421--444.
		
		
		
		
		
		\bibitem{MFGZ} M. Fuchs, G. Zhang, Liouville theorems for entire local minimizers of energies defined on the class $LlogL$ and for entire solutions of the stationary Prandtl-Eyring fluid model, {\it Calc. Var. Partial Differential Equations}, \textbf{44}(2012), no. 1--2, 271--295. 
		
		
		\bibitem{GAGP11}G. P. Galdi, An introduction to the mathematical theory of the Navier-Stokes equations. Steady-state problems, Second edition, Springer Monographs in Mathematics. Springer, New York, 2011.
		
		
		\bibitem{GWASP78}D. Gilbarg, H. F. Weinberger, Asymptotic properties of steady plane solutions of
		the Navier-Stokes equations with bounded Dirichlet integral, {\it Ann. Scuola Norm. Sup. Pisa Cl. Sci. (4)}, {\bf 5}(1978), no. 2, 381--404.
		
		
		
		
		
		\bibitem{BK} B.J. Jin, K. Kang, Liouville theorem for the steady-state non-Newtonian Navier-Stokes equations in two dimensions, {\it J. Math. Fluid Mech.}, \textbf{16}(2014), 275--292.
		
		
		
		
		\bibitem{JKKTARMA66}J. K. Knowles, On Saint-Venant's principle in the two-dimensional linear theory of elasticity, {\it Arch. Ration. Mech. Anal.}, {\bf 21}(1966), 1--22.
		
		
		\bibitem{KNSS09}G. Koch, N. Nadirashvili, G. A. Seregin,  V. Sver\'{a}k, Liouville theorems for the Navier-Stokes equations and applications, {\it Acta Math.}, {\bf 203}(2009), no. 1, 83--105.
		
		\bibitem{OAVA}V. A. Kondrat’ev, O. A. Oleinik, Boundary-value problems for the system of elasticity theory in unbounded domains. Korn’s inequalities, Russian Mathematical Surveys, {\bf 43}(1988), Issue \textbf{5}, 65--119.
		
		
		
		
		
		
		
		
		
		
		
		
		\bibitem{KPRJMFM15} M. Korobkov, K. Pileckas,  R. Russo,
		The Liouville theorem for the steady-state Navier-Stokes problem for axially symmetric 3D solutions in absence of swirl,
		{\it J. Math. Fluid Mech.}, {\bf17}(2015), 287--293.
		
		\bibitem{KTWJFA17}H. Kozono, Y. Terasawa,  Y. Wakasugi, A remark on Liouville-type theorems for the stationary Navier-Stokes equations in three space dimensions,  {\it J. Funct. Anal.}, {\bf 272} (2017), no. 2, 804--818.
		
		
		
		\bibitem{KTWJFM24} H. Kozono, Y. Terasawa,  Y. Wakasugi, Liouville-type theorems for the  Taylor-Couette-Poiseuille flow of the stationary Navier-Stokes equations,  {\it J. Fluid Mech.},  {\bf 989}(2024),  A7.
		
		
		
		
		\bibitem{OAL1} O.A. Ladyzhenskaya, On some new equations for the description of the viscous incompressible fluids and global solvability in the range of the boundary value problems to these equations, {\it Trudy Steklov’s Math. Inst.}, \textbf{102}(1967), 85--104. 
		
		\bibitem{OAL2} O.A. Ladyzhenskaya, On some modifications of the Navier-Stokes equations for large gradients of velocity, {\it Zapiski
			Naukhnych Seminarov LOMI}, \textbf{7}(1968), 126--154. 
		
		\bibitem{OAL3} O.A. Ladyzhenskaya, The Mathematical Theory of Viscous Incompressible Flows. Gordon and Beach, New York (1969).
		
		
		
		
		
		
		
		
		\bibitem{LSZNSL80} O.A. Ladyzhenskaya, V. A. Solonnikov,  Determination of solutions of boundary value problems for stationary Stokes and Navier-Stokes equations having an unbounded Dirichlet integral, {\it Zap. Nauchn. Sem. Leningrad. Otdel. Mat. Inst. Steklov. (LOMI)}, {\bf 96}(1980), 117--160.
		
		
		
		
		\bibitem{LZ} C.M. Li, K. Zhang, A note on the Gagliardo-Nirenberg inequality in a bounded domain, {\it Commun. Pure  Appl. Anal.}, \textbf{21}(2022), no. 12, 4013--4017. 
		
		
		\bibitem{LJL} J. L. Lions, Quelques methodes de resolution des preblemes aux limites non lineaires. Dunod, Gauthier-Villars (1969).
		
		\bibitem{MP} M. Pokorný, Cauchy problem for the non-Newtonian viscous incompressible fluid,     {\it Appl. Math.}, \textbf{41} (1996), no. 3, 169--201.
		
		
		
		
		\bibitem{MR} M. Ruzi\v{c}k\v{a}, A note on steady flow of fluids with shear dependent viscosity, {\it Nonlinear Anal.}, \textbf{30}(1997), 3029--3039.
		
		
		
		
		
		
		
		
		
		
		
		\bibitem{SN16}G. Seregin, Liouville type theorem for stationary Navier-Stokes equations, {\it Nonlinearity}, {\bf 29}(2016), no. 8, 2191--2195.
		
		
		
		
		
		
		
		
		
		
		\bibitem{RATARMA65} R. A. Toupin, Saint-Venant's principle, {\it Arch. Ration. Mech. Anal.}, {\bf 18}(1965), 83--96.
		
		
		
		\bibitem{TT18}T. P. Tsai, Lectures on Navier-Stokes Equations. Graduate Studies in Mathematics,
		192. American Mathematical Society, Providence, RI, 2018.
		
		
		\bibitem{TT21} T. P. Tsai, Liouville type theorems for stationary Navier-Stokes equations, {\it Partial Differ. Equ. Appl.}, {\bf 2}(2021), no. 1, Paper No. 10, 20 pp.
		
		\bibitem{WJDE19} W. Wang,
		Remarks on Liouville type theorems for the 3D steady axially symmetric Navier-Stokes equations,
		{\it J. Differential Equations}, {\bf266}(2019), 6507--6524.
		
		\bibitem{WL} W.L. Wilkinson, Non-Newtonian Fluids. Fluid Mechanics, Mixing and Heat Transfer, Pergamon Press, London, 1960.
		
		\bibitem{yangYin18} J. Yang,
		H. Yin, On the steady non-Newtonian fluids in domains with noncompact boundaries,
		{\it SIAM J. Math. Anal.}, {\bf50}(2018), no. 1, 283--338.
		
		
		\bibitem{GZ} G. Zhang, A note on Liouville theorem for stationary flows of shear thickening fluids in the plane, {\it J. Math. Fluid Mech.}, {\bf15}(2013), no. 4, 771--782.
		
		\bibitem{NZ19} N. Zhao, A Liouville type theorem for axially symmetric D-solutions to steady Navier-Stokes equations, {\it Nonlinear Anal.}, {\bf 187}(2019), 247--258.
		
		
		
		
		
		
	\end{thebibliography}
\end{document}